\documentclass[hidelinks,onefignum,onetabnum]{siamart251216}

\usepackage{microtype}
\usepackage{booktabs}
\usepackage{enumitem}
\usepackage{amssymb}
\usepackage{array}
\setlist[itemize]{topsep=0pt, partopsep=0pt}

\usepackage{lipsum}
\usepackage{amsfonts}
\usepackage{graphicx}
\usepackage{epstopdf}
\usepackage{algorithmic}
\usepackage{subcaption}
\usepackage{bm}
\ifpdf
  \DeclareGraphicsExtensions{.eps,.pdf,.png,.jpg}
\else
  \DeclareGraphicsExtensions{.eps}
\fi

\newsiamremark{remark}{Remark}
\newsiamremark{hypothesis}{Hypothesis}
\crefname{hypothesis}{Hypothesis}{Hypotheses}
\newsiamthm{claim}{Claim}
\newsiamremark{fact}{Fact}
\crefname{fact}{Fact}{Facts}

\headers{Polarized Ensemble Kalman Inversion}{H. Yang, J.-L. Wu.}

\title{Polarized Ensemble Kalman Inversion: \\A Derivative-Free Method for Multimodal Inverse Problems with Non-asymptotic Analysis\thanks{Submitted to the editors DATE.
\funding{This work was supported by the U.S. National Science Foundation under Award No. 2615744 and by the Wisconsin Alumni Research Foundation.}}}

\author{
Huchen Yang\thanks{Department of Mechanical Engineering, University of Wisconsin--Madison, Madison, WI
  (\email{huchen.yang@wisc.edu}, \email{jinlong.wu@wisc.edu}).}
\and
Jin-Long Wu \footnotemark[2]
}

\usepackage{amsopn}

\ifpdf
\hypersetup{
  pdftitle={Polarized Ensemble Kalman Inversion: A Derivative-Free Method for Multimodal Inverse Problems with Non-asymptotic Analysis},
  pdfauthor={H. Yang, JL. Wu.}
}
\fi

\usepackage{xr-hyper}
\begin{document}

\maketitle

\begin{abstract}
Inverse problems often admit multiple solutions consistent with the same noisy, indirect observations, especially with sparse data and/or strongly nonlinear forward models. Ensemble Kalman inversion (EKI) provides a robust, derivative-free framework to solve inverse problems, but its global empirical covariances couple all ensemble members across distinct solution modes, limiting its effectiveness for multimodal problems. We introduce polarized ensemble Kalman inversion (PEKI), which constructs distance-weighted local covariances for each ensemble member. This localization weakens contributions from distant, data-consistent regions and enables different clusters of ensemble members to evolve toward distinct solution modes. For the practical algorithm with finite ensemble size and polarization radius, we also establish non-asymptotic bounds and perform numerical studies, showing that cross-cluster effects decay exponentially with cluster separation relative to the polarization radius. For linear forward maps, we further establish particle-wise residual contraction and derive an exact cluster-mean decomposition that separates cross-cluster leakage, within-cluster gain variation, and observation perturbations. These results further motivate a one-step procedure for choosing the polarization radius by reusing stored forward-model evaluations. We studied the performance of PEKI on several numerical experiments, including low-dimensional nonconvex benchmarks, a problem equipped with many-mode landscape, and a high-dimensional inverse problem governed by Navier--Stokes equations. The results confirm the theoretical exponential decay of cross-cluster effects and show that PEKI can serve as a robust and derivative-free method for inverse problems with multiple solution modes.
\end{abstract}



\section{Introduction}

Many scientific computing tasks involve solving inverse problems, e.g., when unknown parameters and/or states are inferred from observations and knowledge of a forward model~\cite{stuart2010inverse}. However, the observed data may be consistent with multiple solutions of those unknown parameters and states, possibly because of sparse or incomplete observations, a highly nonlinear, reduced-order, or symmetric forward model, or simply different physical mechanisms that lead to similar system behavior~\cite{beven2001equifinality,kaipio2007statistical,eldredge2026practical}. Typical examples include locating acoustic or pollutant sources from sparse sensor measurements~\cite{keats2007bayesian}, recovering scatterer locations or shapes from wave measurements~\cite{colton2003inverse}, inferring subsurface permeability fields from well data~\cite{oliver2008inverse}, and identifying coherent flow structures from pressure or velocity measurements~\cite{eldredge2024bayesian}. Failing to identify multiple data-consistent solutions can mislead subsequent interpretation, prediction, or decision-making~\cite{stuart2010inverse,beven2001equifinality}.

A variety of approaches have been developed for multimodal inverse problems. From a Bayesian perspective, sampling methods such as Markov chain Monte Carlo (MCMC) and sequential Monte Carlo (SMC) can represent multiple solutions within the posterior distribution~\cite{sambridge2002monte, del2006sequential}, but require many forward-model evaluations and can have difficulty in adequately exploring well-separated modes~\cite{vousden2016dynamic, paulin2019error}. If the goal is to locate multiple solutions rather than characterize the posterior distributions as Bayesian inverse problems, one can formulate the problem as the minimization of a data-misfit or regularized objective, for which gradient-based methods efficiently identify a local minimizer~\cite{haber2000optimization, givoli2021tutorial}. However, their use for certain types of forward models is limited when gradient information is unavailable, expensive to obtain, or noisy and unreliable~\cite{larson2019derivative}. Moreover, identifying multiple solutions via gradient-based methods generally requires repeated optimization from different initial guesses~\cite{marti2025multi}, and the guarantee of identifying all possible solutions often requires strong assumptions on the initialization.

For inverse problems with computationally intensive forward models, ensemble Kalman methods provide an attractive alternative because they are derivative-free and the forward model evaluations can be easily parallelized across ensemble members~\cite{iglesias2013ensemble,kovachki2019ensemble,gjini2026ensemble}. Their standard formulation, however, is not designed to identify multiple solutions~\cite{dovera2011multimodal}. Standard ensemble Kalman inversion (EKI) computes global empirical means and covariance matrices from the entire ensemble and uses the same statistics to update all ensemble members~\cite{grumitt2024flow}. When the ensemble spans several separated solutions, these global statistics can mix information associated with different solutions, causing the resulting update directions to reflect a compromise among them rather than guiding ensemble members toward the distinct solutions.

In this work, we present a systematic study of polarized ensemble Kalman inversion (PEKI), which is built on several recent works~\cite{bungert2025polarized, wacker2026perspectives} and serves as a practical locally weighted ensemble Kalman method for identifying multiple plausible solutions. In addition, we propose an efficient hyperparameter tuning strategy and present non-asymptotic theoretical analysis for finite ensemble size and polarized radius. Specifically, PEKI updates each ensemble member using its own local covariance matrices. By assigning smaller weights to distant ensemble members, PEKI reduces the influence of ensemble members near distant solutions, allowing each ensemble member to evolve primarily toward a nearby solution. Consequently, different clusters of the ensemble members can evolve toward distinct solutions, enabling multiple plausible solutions to be identified within a single ensemble. 


The key highlights of this work are summarized as follows:
\begin{itemize}
\item We develop PEKI as a practical finite-ensemble method for multimodal inverse problems using particle-dependent polarized covariance matrices. In addition, we establish an efficient strategy for selecting the polarization radius that requires no additional forward-model evaluations.
\item We establish non-asymptotic theoretical results for finite ensembles and finite polarization radii. We show that cross-cluster interactions are exponentially suppressed as cluster separation increases relative to the polarization radius. For linear forward maps, we further characterize particle-wise residual non-expansion and derive an exact decomposition of cluster-mean dynamics.
\item We systematically evaluate the performance of PEKI on numerical examples with up to 128 parameters and 64 competing solution modes, showcasing practical polarization-radius selection, providing empirical evidence for the predicted cross-cluster suppression and the sharpness of its leading exponential scale, and demonstrating convergence toward the identified solutions.
\end{itemize}

This paper is organized as follows. Section~\ref{sec:related_work} reviews earlier approaches to localization and multimodality in ensemble Kalman methods and clarifies the relation between PEKI and the existing methods. Section~\ref{sec: Methodology} presents the PEKI formulation and theoretical analysis. Section~\ref{sec: choose} presents the practical polarization-radius selection strategy. Section~\ref{sec: Numerical Results} provides the numerical results. Section~\ref{Conclusion} concludes the paper. 


\section{Related Work}
\label{sec:related_work}

\subsection{Particle-dependent local update in ensemble Kalman methods}

Earlier ensemble filtering studies had already explored member-wise local updates, in which neighboring ensemble members are used to construct a member-specific update~\cite{anderson2003local, metref2014non}. A prominent practical example is the iterative local updating ensemble smoother (ILUES)~\cite{zhang2018iterative}. For each target particle, ILUES constructs a local ensemble using discrepancies in both parameter and observation spaces, applies an ensemble-smoother update to this local ensemble, and then replaces the target particle with a randomly selected member of the updated local ensemble. Subsequent applications have demonstrated that local covariance information can help maintain multiple solutions without prescribing a fixed number of modes~\cite{mo2019deep,zheng2023efficient,zhang2024integration}. At the same time, however, ILUES relies on hard, fixed-cardinality neighborhood selection, which makes the geometric interaction range dependent on the empirical particle density and allows the selected neighborhood to change discontinuously.


Smooth particle-dependent interactions, previously introduced through polarization in interacting particle systems~\cite{bungert2025polarized}, provide an alternative to the hard local-ensemble selection. Closely related particle-dependent localization had previously been introduced for Bayesian sampling through ensemble Kalman sampling (EKS), particularly for multimodal targets~\cite{reich2021fokker}. The same principle was subsequently developed for optimization-based inversion through locally weighted ensemble Kalman inversion (lwEKI)~\cite{wacker2026perspectives}. lwEKI formulates the resulting dynamics as a system of interacting particle ordinary differential equations. Its theoretical analysis primarily concerns the relation between localized ensemble statistics and the local differential structure of the forward map. In particular, it studies a mean-field and vanishing-bandwidth regime in which a suitably localized ensemble derivative approaches the true local derivative of the forward map. It establishes an important theoretical foundation for particle-dependent localization. However, these results do not directly address the finite-computation regime encountered in practical inversion. Practical inversion, by contrast, operates with a finite ensemble, a finite nonzero bandwidth, and finitely many discrete updates. In this regime, a suitable bandwidth must be chosen, but a practical strategy remains unclear.

Motivated by this issue, PEKI starts from the question: how strongly do particle clusters associated with separated solutions influence one another through the localized Kalman update? This leads to a cluster-separation analysis that quantifies such cross-cluster interactions and, in turn, provides practical guidance for selecting the polarization radius.

\subsection{Other ensemble Kalman approaches to multimodality}

Several other lines of work have extended ensemble Kalman methods beyond a single global Gaussian-like ensemble representation. Gaussian-mixture methods explicitly represent multiple modes, with separate Kalman updates performed for the individual mixture components~\cite{dovera2011multimodal,frei2013mixture}. These methods generally require selecting or controlling the number of components and assigning ensemble members to them~\cite{anderson2024quantile,gryvill2025bayesian}, with some additionally requiring particle reweighting or resampling~\cite{chen2024efficient,chan2023potential,dovera2011multimodal}. Related formulations incorporate Kalman updates into variational methods for multimodal inference~\cite{che2025stable,che2026adaptive}. Across these mixture-based approaches, some mixture-component weights may vanish, while multiple component means collapse onto the same posterior mode, thereby limiting the number of distinct modes effectively represented.

A second strategy introduces correction weights~\cite{papadakis2010data,ding2020ensemble} or combines ensemble Kalman updates with particle filters~\cite{frei2013bridging,shen2015modified,subrahmanya2026copula} to preserve non-Gaussian structure. However, in the Kalman-update part of these methods, all ensemble members often still share the same Kalman gain constructed from covariance matrices of the full ensemble. As a result, ensemble members associated with separated solutions may remain coupled through the global covariance~\cite{chustagulprom2016hybrid}.

Other approaches map the parameters into a latent Gaussian space~\cite{grumitt2024flow,pasmans2026ensemble,zhou2026neural} or augment Kalman methods with score-based diffusion models~\cite{zheng2024ensemble,reimann2026robust}, introducing auxiliary representations or learned models beyond the basic ensemble Kalman update.

\section{Methodology}\label{sec: Methodology}

\subsection{General inverse problem}

We consider the recovery of an unknown parameter \(\theta\in\mathbb{R}^{d_\theta}\) from noisy observations \(y\in\mathbb{R}^{d_y}\),
\begin{equation}
    y=\mathcal{G}(\theta^\dagger)+\eta,
    \qquad
    \eta\sim\mathcal{N}(0,\Gamma),
    \label{eq:observation_model}
\end{equation}
where \(\theta^\dagger\) is the true parameter,
\(\mathcal{G}:\mathbb{R}^{d_\theta}\to\mathbb{R}^{d_y}\) is the forward map, and \(\Gamma\) is the observational noise covariance. With prior density
\(\pi_0\), the posterior density is
\begin{equation}
    \pi^y(\theta)
    =\frac{1}{Z(y)}\exp\bigl(-\Phi(\theta;y)\bigr)\pi_0(\theta),
    \qquad
    \Phi(\theta;y)
    =\frac12\left\|\Gamma^{-1/2}
    \bigl(\mathcal{G}(\theta)-y\bigr)\right\|^2,
    \label{eq:posterior_density}
\end{equation}
where \(Z(y)\) is the normalizing constant and $\Phi(\theta;y)$ is the data-misfit function. Within this Bayesian formulation, EKI provides an ensemble-based, derivative-free approach to the inverse problem by iteratively seeking parameters with reduced data misfit~\cite{kovachki2019ensemble},
\[\theta^\ast \in \operatorname*{arg\,min}_{\theta}\Phi(\theta;y),\]
rather than by directly representing the posterior distribution. 

For nonlinear forward maps, several separated parameter values may be consistent with the data, resulting in a multimodal inverse problem. In this setting, a globally Kalman-updated ensemble can be misleading: it would retain only one candidate solution, or its mean can lie between several solutions. Our objective is therefore to enable a single ensemble to identify multiple data-consistent solutions.

\subsection{Polarized ensemble Kalman inversion}

The central modification in polarized ensemble Kalman inversion (PEKI) is to replace the global empirical covariances of EKI with covariances specialized for each target particle. Throughout this paper, we use the terms \emph{particle} and \emph{ensemble member} interchangeably. Let \(\{\theta_n^{(j)}\}_{j=1}^J\) be the ensemble at iteration \(n\), let \(g_n^{(j)}=\mathcal{G}(\theta_n^{(j)})\), and let \(N\) be the total number of iterations. For a target particle \(\theta_n^{(i)}\), define
\[
\begin{aligned}
    d_{ij,n}^2
    &:=\left\|\theta_n^{(i)}-\theta_n^{(j)}\right\|_{M_n}^2
      =\bigl(\theta_n^{(i)}-\theta_n^{(j)}\bigr)^\top
       M_n\bigl(\theta_n^{(i)}-\theta_n^{(j)}\bigr),
    \\
    \widetilde w_{ij,n}
    &:=\exp\left(-\frac{d_{ij,n}^2}{2\kappa^2}\right),
    \qquad
    w_{ij,n}
    :=\frac{\widetilde w_{ij,n}}
    {\sum_{\ell=1}^J\widetilde w_{i\ell,n}},
    \label{eq:peki_weights}
\end{aligned}
\]
where \(\kappa>0\) is the polarization radius and \(M_n\succ0\) specifies the distance metric. We use \(M_n=I\) unless a problem-dependent scaling is needed. For example, a Gaussian prior \(\mathcal{N}(m_0,C_0)\) naturally suggests \(M_n=C_0^{-1}\).

The corresponding local means and covariance denominator are
\[\bar\theta_{i,n}:=\sum_{j=1}^Jw_{ij,n}\theta_n^{(j)},
    \qquad
    \bar g_{i,n}:=\sum_{j=1}^Jw_{ij,n}g_n^{(j)},
    \qquad
    D_{i,n}:=1-\sum_{j=1}^Jw_{ij,n}^2.\]
Whenever \(D_{i,n}>0\), the particle-dependent cross-covariance and data covariance are
\begin{equation}\label{eq:peki_cov_theta_g}
    \begin{aligned}
    C_{i,n}^{\theta g}
    &:=\frac{1}{D_{i,n}}
    \sum_{j=1}^Jw_{ij,n}
    \bigl(\theta_n^{(j)}-\bar\theta_{i,n}\bigr)
    \bigl(g_n^{(j)}-\bar g_{i,n}\bigr)^\top,
    \\
    C_{i,n}^{gg}
    &:=\frac{1}{D_{i,n}}
    \sum_{j=1}^Jw_{ij,n}
    \bigl(g_n^{(j)}-\bar g_{i,n}\bigr)
    \bigl(g_n^{(j)}-\bar g_{i,n}\bigr)^\top.
\end{aligned}
\end{equation}
Define the particle-dependent Kalman gain
\[
    K_{i,n}:=C_{i,n}^{\theta g}
    \bigl(C_{i,n}^{gg}+R_N\bigr)^{-1}, \quad R_N:=N\Gamma+\lambda I,
\]
where \(\lambda\geq0\) is a regularization parameter, and assume throughout that \(R_N\succ0\).
PEKI then applies the update
\begin{equation}
    \theta_{n+1}^{(i)}
    =\theta_n^{(i)}
    +K_{i,n}
    \bigl(y_n^{(i)}-g_n^{(i)}\bigr),
    \qquad
    y_n^{(i)}=y+\xi_n^{(i)},
    \quad
    \xi_n^{(i)}\sim\mathcal{N}(0,N\Gamma),
    \label{eq: peki update}
\end{equation}
for \(i=1,\ldots,J\). In the deterministic formulation, \(\xi_n^{(i)}=0\).

After \(N\) iterations, separated clusters in the final ensemble, together with their empirical means, provide representative candidate solutions. Various clustering methods can be used for this post-processing step; in the numerical examples, we use HDBSCAN~\cite{campello2015hierarchical, McInnes2017} to identify the particle clusters.

The polarization radius \(\kappa\) is the only hyperparameter that directly controls polarization: increasing \(\kappa\) enlarges the interaction range and makes PEKI closer to a global EKI update, whereas decreasing \(\kappa\) reduces interaction between distant particles. The ensemble size \(J\), number of iterations \(N\), and metric \(M_n\) are treated as algorithmic or numerical settings rather than additional polarization hyperparameters.

The particle-dependent weighting in PEKI is different from conventional localization in ensemble Kalman methods. Conventional localization acts on correlations between parameter or state components, whereas PEKI weights the contributions of different ensemble members according to their distances in parameter space. Moreover, the distances used in conventional localization are typically fixed by the problem geometry, while the particle-to-particle distances in PEKI evolve with the ensemble. 

\subsection{Properties of PEKI}

The analytical results are as follows. First, PEKI approaches global EKI as \(\kappa\to\infty\), whereas its \(\kappa\to0\) limit is an update based only on the target particle and its nearest neighbors. Second, interactions between particle clusters are progressively suppressed as their separation increases relative to the polarization radius \(\kappa\), leading to increasingly decoupled PEKI updates. Third, for a linear forward map, each PEKI particle retains the standard EKI residual non-expansion mechanism. Finally, the cluster-mean update admits an exact decomposition into an ordinary cluster-EKI update, a cross-cluster correction, a within-cluster correction, and an observation-perturbation term. A concise proof of the cross-cluster suppression result is provided in Appendix \ref{app:cluster_decoupling_bounds}, while its full proof, together with proofs for the other results, is given in the Supplement.

\subsubsection{Limiting regimes}
\label{subsec:algebraic_structure_limiting_regimes}

We first characterize the two limiting regimes of the polarization radius $\kappa$.


\begin{proposition}[Global EKI limit]\label{prop:global_eki_limit}
Assume \(J\geq2\) and fix the ensemble at iteration \(n\). As \(\kappa\to\infty\),
\[
\begin{gathered}
    w_{ij,n}\longrightarrow\frac1J,
    \qquad D_{i,n}=1-\sum_{j=1}^Jw_{ij,n}^2\longrightarrow\frac{J-1}{J}, \\
    \bar\theta_{i,n}\longrightarrow\bar\theta_n :=\frac{1}{J}\sum_{j=1}^J\theta_n^{(j)}, \qquad \bar g_{i,n}\longrightarrow\bar g_n:=\frac{1}{J}\sum_{j=1}^Jg_n^{(j)},
\end{gathered}
\]
Therefore, the particle-dependent covariance matrices become independent of the target particle and converge to the following standard global empirical covariances: 
\[\begin{aligned}
    C_{i,n}^{\theta g}
\longrightarrow C_n^{\theta g,\mathrm{glob}}
    &:=\frac1{J-1}\sum_{j=1}^J
    (\theta_n^{(j)}-\bar\theta_n)(g_n^{(j)}-\bar g_n)^\top,
    \\
   C_{i,n}^{g g}
\longrightarrow C_n^{gg,\mathrm{glob}}
    &:=\frac1{J-1}\sum_{j=1}^J
    (g_n^{(j)}-\bar g_n)(g_n^{(j)}-\bar g_n)^\top.
\end{aligned}\]
Consequently, the PEKI update converges to the standard stochastic EKI update formed from these global covariance matrices.
 
\end{proposition}

\begin{proposition}[Small-\(\kappa\) nearest-neighbor limit]
\label{prop:small_kappa_nearest_neighbor}
Fix particle \(i\) at iteration \(n\), and assume that no other particle coincides with particle $i$, i.e., $d_{ij,n}>0,\,j\neq i$. Define the distance from particle $i$ to its nearest distinct particle(s), the corresponding nearest-neighbor index set, and the number of such nearest neighbors by
\[
    r_{i,n}:=\min_{j\neq i}d_{ij,n},
    \qquad
    \mathcal{N}_{i,n}:=
    \{j\neq i:d_{ij,n}=r_{i,n}\},
    \qquad
    m_i:=|\mathcal{N}_{i,n}|.
\]
Then, as \(\kappa\to0\),
\[\begin{aligned}
    C_{i,n}^{\theta g}
    &\longrightarrow
    \frac1{2m_i}\sum_{j\in\mathcal{N}_{i,n}}
    (\theta_n^{(i)}-\theta_n^{(j)})
    (g_n^{(i)}-g_n^{(j)})^\top,
    \\
    C_{i,n}^{gg}
    &\longrightarrow
    \frac1{2m_i}\sum_{j\in\mathcal{N}_{i,n}}
    (g_n^{(i)}-g_n^{(j)})
    (g_n^{(i)}-g_n^{(j)})^\top.
\end{aligned}\]
Thus, in the small-$\kappa$ limit, the particle-dependent covariance information is determined entirely by particle $i$ and its nearest neighbor(s). If the nearest neighbor is unique, the limiting covariance is exactly the covariance obtained from the two-particle pair consisting of particle $i$ and that neighbor. If several particles are tied at the same nearest-neighbor distance, the limit is the average of the corresponding pairwise secant covariances.

Consequently, the PEKI update of particle i is driven entirely by its nearest neighbor(s), while all more distant particles become asymptotically irrelevant. In the unique-nearest-neighbor case, the limiting update is exactly the update that particle i would receive from an EKI applied to the two-particle ensemble consisting of itself and its nearest neighbor.

Intuitively, as $\kappa \to 0$, each particle assigns almost all of its weight to itself, so the uncorrected weighted covariance would vanish. However, the covariance denominator $D_{
i,n}$ vanishes at the same order, and the corrected covariance therefore has a finite nearest-neighbor limit.

\end{proposition}

\begin{remark}[Small-\(\kappa\) update behavior]
\label{rem:small_kappa_numerical_implications}
As \(\kappa\to0\), each particle assigns almost all of its weight to itself: \(w_{ii,n}\to1\) and \(w_{ij,n}\to0\) for \(j\neq i\). Without the finite-sample correction by $D_{i,n}$, the corresponding weighted covariance would therefore vanish. However, the covariance denominator $D_{i,n}$ vanishes at the same order, so the corrected covariance has the finite nearest-neighbor limit established above. Consequently, PEKI does not generally stop updating as \(\kappa\to0\).

Strong polarization therefore does not eliminate the update, but restricts it to information from the nearest neighbor(s). If these neighbors are close to the target particle and have similar forward-model outputs, the resulting covariance and Kalman gain can be small, leading to only a small particle displacement. Thus, the update may appear nearly stationary even though its small-$\kappa$ limit is generally nonzero.
\end{remark}

\begin{remark}[PEKI subspace property]\label{prop:local_subspace_property}
For every particle, the PEKI update satisfies
\[
    \theta_{n+1}^{(i)}-\theta_n^{(i)}
    \in
    \operatorname{span}\left\{
    \theta_n^{(j)}-\bar\theta_{i,n}:w_{ij,n}>0
    \right\}.
\]
For the Gaussian kernel, all weights are positive, so this is algebraically the
full ensemble span. Nevertheless, the contribution of a particle to the covariance and update decreases exponentially with its distance from the target particle (quantified in Section~\ref{subsec:cluster_decoupling}). Hence, in practical computation, the update is primarily governed by directions associated with nearby particles. In this sense, the subspace is effectively local even though it is algebraically global. With a compactly supported or thresholded kernel, distant particles have exactly zero weight, and the same relation becomes a literal local-subspace property.
\end{remark}

\subsubsection{Suppression of cross-cluster interactions}
\label{subsec:cluster_decoupling}

This subsection shows that cross-cluster interactions are progressively suppressed as the separation between particle clusters increases relative to $\kappa$. Specifically, the exponential suppression of cross-cluster weights propagates to the local covariance matrices and Kalman gains, so that the PEKI update of a particle approaches the update it would receive if all particles outside its cluster were removed.

To compare the full PEKI update with the update obtained using only particles
within a single cluster, we first introduce the corresponding
cluster-restricted PEKI quantities. Let \(I_1,\ldots,I_M\) partition \(\{1,\ldots,J\}\). Fix \(i\in I_r\), and let \(J_r:=|I_r|\) denote the number of particles in cluster \(I_r\), with \(2\le J_r<J\). Define the cluster-restricted weights
\[
    w_{ij,n}^{(r)}
    :=
    \frac{\widetilde w_{ij,n}}
    {\sum_{\ell\in I_r}\widetilde w_{i\ell,n}},
    \qquad j\in I_r.
\]
The cluster-restricted means and covariances \(\bar\theta_{i,n}^{(r)}\), \(\bar g_{i,n}^{(r)}\), \(C_{i,r,n}^{\theta g}\), and \(C_{i,r,n}^{gg}\) are defined by the same formulas as their full-ensemble counterparts, with sums restricted to \(I_r\) and weights replaced by \(w_{ij,n}^{(r)}\). For later use, we explicitly write the cluster-restricted covariance denominator and Kalman gain as
\[
    D_{i,r,n}
    :=
    1-\sum_{j\in I_r}\bigl(w_{ij,n}^{(r)}\bigr)^2,
    \qquad
    K_{i,r,n}
    :=
    C_{i,r,n}^{\theta g}
    \bigl(C_{i,r,n}^{gg}+R_N\bigr)^{-1}.
\]
The update using only particles in \(I_r\) is
\[
    \widetilde\theta_{n+1}^{(i,r)}
    :=
    \theta_n^{(i)}
    +K_{i,r,n}
    \bigl(y_n^{(i)}-g_n^{(i)}\bigr).
\]

To quantify the discrepancy between the full and cluster-restricted PEKI
quantities, define the \emph{cross-cluster leakage} as the total normalized weight assigned by particle \(i\) to particles outside its cluster, and the nearest outside-cluster distance
\[
    \varepsilon_{i,r,n}
    :=
    \sum_{j\notin I_r}w_{ij,n},
    \qquad
    d_{i,r,n}^{\mathrm{out}}
    :=
    \min_{j\notin I_r}d_{ij,n}.
\]

\begin{proposition}[Exponential suppression of cross-cluster weights]
\label{prop:cross_cluster_weight_suppression}
The cross-cluster leakage satisfies
\[
    \varepsilon_{i,r,n}
    \le
    \frac{(J-J_r)}{(J-J_r)+\exp\left[
    (d_{i,r,n}^{\mathrm{out}})^2/2\kappa^2
    \right]}\le
    (J-J_r)
    \exp\left(
    -\frac{(d_{i,r,n}^{\mathrm{out}})^2}{2\kappa^2}
    \right).
\]
\end{proposition}

\begin{proposition}[Covariance and Kalman-gain decoupling]
\label{prop:covariance_gain_decoupling}
Let
\[
    R_{\theta,n}
    :=
    \max_{p,q}\|\theta_n^{(p)}-\theta_n^{(q)}\|,
    \qquad
    R_{g,n}
    :=
    \max_{p,q}\|g_n^{(p)}-g_n^{(q)}\|
\]
denote the ensemble diameters in parameter and observation spaces, respectively,
and let
\[
    \mu_N:=\lambda_{\min}(R_N)>0,
    \qquad
    B_n
    :=
    \frac{2R_{\theta,n}R_{g,n}}{\mu_N}
    +
    \frac{R_{\theta,n}R_{g,n}^3}{\mu_N^2}.
\]
be the smallest eigenvalue of \(R_N\) and the shorthand bound constant.

The full and cluster-restricted covariances satisfy
\begin{align}
    \left\|
        C_{i,n}^{\theta g}-C_{i,r,n}^{\theta g}
    \right\|
    &\le
    \frac{2R_{\theta,n}R_{g,n}}
    {(1-\varepsilon_{i,r,n})^2D_{i,r,n}}
    \varepsilon_{i,r,n}, \label{eq:c_theta_g_bound}
    \\
    \left\|
        C_{i,n}^{gg}-C_{i,r,n}^{gg}
    \right\|
    &\le
    \frac{2R_{g,n}^2}
    {(1-\varepsilon_{i,r,n})^2D_{i,r,n}}
    \varepsilon_{i,r,n},\label{eq:c_g_g_bound}
\end{align}
and the corresponding Kalman gains satisfy
\begin{equation}
    \left\|
        K_{i,n}-K_{i,r,n}
    \right\|
    \le
    B_n
    \frac{\varepsilon_{i,r,n}}
    {(1-\varepsilon_{i,r,n})^2D_{i,r,n}}.
    \label{eq:covariance_gain_decoupling}
\end{equation}
Consequently, provided that \(D_{i,r,n}\) remains nondegenerate, these
discrepancies are exponentially suppressed as
\(d_{i,r,n}^{\mathrm{out}}/\kappa\) increases.
\end{proposition}

\begin{corollary}[Exponentially small one-step cluster-update discrepancy]
\label{cor:cluster_update_discrepancy}
The full and cluster-restricted updates satisfy a non-asymptotic bound:
\[
\begin{aligned}
\left\|
\theta_{n+1}^{(i)}
-\widetilde\theta_{n+1}^{(i,r)}
\right\|
&\le
B_n
\frac{\varepsilon_{i,r,n}}
{(1-\varepsilon_{i,r,n})^2D_{i,r,n}}
\left\|y_n^{(i)}-g_n^{(i)}\right\|\\
&\le
\frac{B_n}{D_{i,r,n}}
(J-J_r)
\exp\left(
-\frac{(d_{i,r,n}^{\mathrm{out}})^2}{2\kappa^2}
\right)\\
&\qquad\times
\left[
1+
(J-J_r)
\exp\left(
-\frac{(d_{i,r,n}^{\mathrm{out}})^2}{2\kappa^2}
\right)
\right]
\left\|y_n^{(i)}-g_n^{(i)}\right\|.
\end{aligned}
\]

\end{corollary}

Thus, provided that the within-cluster covariance denominator remains nondegenerate, the full PEKI update approaches the update obtained after removing all particles outside \(I_r\) as \(d_{i,r,n}^{\mathrm{out}}/\kappa\) increases. 

Together, these bounds provide a quantitative interpretation of the polarization radius \(\kappa\): for a fixed separation between particle clusters, smaller values of \(\kappa\) more strongly suppress cross-cluster interactions, whereas larger values permit stronger interaction between the clusters. The numerical results in the two examples further show that this exponential dependence closely captures the observed decay of cross-cluster interactions, suggesting that the predicted scale is empirically sharp up to problem-dependent prefactors.

\subsubsection{Residual contraction in the linear case}
\label{subsec:residual_contraction_local_convergence}

For a linear forward map, each PEKI particle retains the standard EKI residual non-expansion property, with the global covariance of standard EKI replaced by its particle-dependent local covariance. Consequently, the residual contraction has the same mathematical form as in standard EKI, although its contraction factor depends on the local covariance.

\begin{proposition}[Particle-wise residual non-expansion]
Assume \(\mathcal{G}(\theta)=A\theta\) and define the local parameter covariance
\[
    C_{i,n}^{\theta\theta}
    :=\frac1{D_{i,n}}\sum_{j=1}^Jw_{ij,n}
    (\theta_n^{(j)}-\bar\theta_{i,n})
    (\theta_n^{(j)}-\bar\theta_{i,n})^\top.
\]
Then
\[
    C_{i,n}^{\theta g}=C_{i,n}^{\theta\theta}A^\top,
    \qquad
    C_{i,n}^{gg}=AC_{i,n}^{\theta\theta}A^\top,
    \qquad
    K_{i,n}=C_{i,n}^{\theta\theta}A^\top
    \bigl(AC_{i,n}^{\theta\theta}A^\top+R_N\bigr)^{-1}.
\]
For a fixed observation vector \(\widetilde y^{(i)}\) (perturbed or unperturbed), let
\[
    \rho_{i,n}:=\widetilde y^{(i)}-A\theta_n^{(i)},
    \qquad
    S_{i,n}:=R_N^{-1/2}AC_{i,n}^{\theta\theta}A^\top R_N^{-1/2}\succeq0.
\]
be the residual of particle $i$ and the corresponding scaled local covariance in the observation space.

After one PEKI update, the residual satisfies
\[
    R_N^{-1/2}\rho_{i,n}^{+}
    =(I+S_{i,n})^{-1}R_N^{-1/2}\rho_{i,n},
    \label{eq:peki_residual_identity_from_eki}
\]
and hence
\begin{equation}
    \|\rho_{i,n}^{+}\|_{R_N^{-1}}
    \leq
    \frac{1}{1+\lambda_{\min}(S_{i,n})}
    \|\rho_{i,n}\|_{R_N^{-1}}
    \leq
    \|\rho_{i,n}\|_{R_N^{-1}},
    \label{eq:peki_residual_contraction_from_eki}
\end{equation}
where \(\|z\|_{R_N^{-1}}:=(z^\top R_N^{-1}z)^{1/2}\).
\end{proposition}

We next consider a different question: how much is this particle-wise residual contraction affected by particles outside the target particle's own cluster? To isolate this cross-cluster effect, we next compare the residual contraction of the full and cluster-restricted PEKI updates.

\begin{proposition}[Particle-wise residual contraction under cluster separation]
\label{prop:cluster_wise_residual_contraction}

Let \(C_{i,r,n}^{\theta\theta}\) be the cluster-restricted parameter covariance by restricting the preceding weighted covariance to \(I_r\), and let $S_{i,r,n}:=R_N^{-1/2}AC_{i,r,n}^{\theta\theta}A^\top R_N^{-1/2}$ denote the corresponding scaled covariance in the observation space.

Assume
\begin{equation}
    S_{i,r,n}\succeq s_0I
    \label{eq:cluster_restricted_strict_contraction}
\end{equation}
for some \(s_0>0\), and define
\[
    \eta_{i,r,n}
    :=\frac{2\|R_N^{-1/2}A\|_2^2R_{\theta,n}^2}
    {(1-\varepsilon_{i,r,n})^2D_{i,r,n}}
    \varepsilon_{i,r,n}.
\]
Then
\begin{equation}
    S_{i,n}\succeq(s_0-\eta_{i,r,n})I.
    \label{eq:S_lower_bound_leakage}
\end{equation}
If \(\eta_{i,r,n}<s_0\),
\begin{equation}
    \|\rho_{i,n}^{+}\|_{R_N^{-1}}
    \leq
    \frac{1}{1+s_0-\eta_{i,r,n}}
    \|\rho_{i,n}\|_{R_N^{-1}}.
    \label{eq:cluster_wise_residual_contraction_factor}
\end{equation}
Thus, cross-cluster leakage weakens the residual contraction by reducing the effective contraction margin from $s_0$ to $s_0-\eta_{i,r,n}$. As $\kappa$ decreases, cross-cluster leakage is increasingly suppressed, thereby reducing this source of degradation.
\end{proposition}

\begin{remark}[Perturbed observations]
\label{rem:perturbed_observation_one_step_scope}
The one-step identities above remain valid when conditioning on a fixed realization of a perturbed observation. If a new perturbed observation is drawn at every iteration, however, the target changes between updates, so they do not by themselves imply deterministic multi-step contraction toward the fixed observation \(y\).
\end{remark}

\begin{remark}[Meaning of the strict-contraction assumption]
\label{rem:strict_contraction_rank_condition}
The condition \eqref{eq:cluster_restricted_strict_contraction} is a pointwise sufficient condition that the cluster-restricted covariance excites every observation direction. If it fails, for example because this covariance is rank deficient, \eqref{eq:peki_residual_contraction_from_eki} still gives non-expansion but not strict contraction in every direction. No uniform-in-time geometric rate is claimed.
\end{remark}

\subsubsection{Cluster-mean dynamics in the linear case}
\label{subsec:exact_peki_cluster_mean_dynamics}

The particle-wise results above do not directly describe cluster-mean dynamics. We therefore derive an exact decomposition of the cluster-mean update into a cluster-EKI update, a cross-cluster correction, a within-cluster correction, and an observation-perturbation term.

For a cluster \(I_r\) with \(J_r:=|I_r|\geq2\), define its empirical parameter mean and ordinary empirical parameter covariance by
\[
    m_{r,n}:=\frac1{J_r}\sum_{i\in I_r}\theta_n^{(i)},
    \qquad
    \widehat C_{r,n}^{\theta\theta}
    :=\frac1{J_r-1}\sum_{i\in I_r}
    (\theta_n^{(i)}-m_{r,n})(\theta_n^{(i)}-m_{r,n})^\top.
\]
The Kalman gain of standard EKI applied to cluster $I_r$, hereafter the cluster-EKI gain, is
\[
    K_{r,n}^{\mathrm{cl}}
    :=\widehat C_{r,n}^{\theta\theta}A^\top
    \bigl(A\widehat C_{r,n}^{\theta\theta}A^\top+R_N\bigr)^{-1}.
\]
For each \(i\in I_r\), recall the PEKI gain \(K_{i,n}\) and the cluster-restricted PEKI gain \(K_{i,r,n}\). Write the observation as \(y_n^{(i)}=y+\xi_n^{(i)}\), with \(\xi_n^{(i)}=0\) in the deterministic formulation.

\begin{proposition}[Exact decomposition of the PEKI cluster-mean update]
\label{prop:exact_peki_cluster_mean_decomposition}
For \(\mathcal{G}(\theta)=A\theta\), the cluster mean satisfies
\begin{equation}
\begin{aligned}
    m_{r,n+1}
    &=m_{r,n}
    +K_{r,n}^{\mathrm{cl}}(y-Am_{r,n})
    +\delta_{r,n}^{\mathrm{out}}
    +\delta_{r,n}^{\mathrm{in}}
    +\nu_{r,n},
\end{aligned}
\label{eq:exact_peki_cluster_mean_decomposition}
\end{equation}
where
\[\begin{aligned}
    \delta_{r,n}^{\mathrm{out}}
    &:=\frac1{J_r}\sum_{i\in I_r}
    (K_{i,n}-K_{i,r,n})(y-A\theta_n^{(i)}),
    \\
    \delta_{r,n}^{\mathrm{in}}
    &:=\frac1{J_r}\sum_{i\in I_r}
    (K_{i,r,n}-K_{r,n}^{\mathrm{cl}})(y-A\theta_n^{(i)}),
    \\
    \nu_{r,n}
    &:=\frac1{J_r}\sum_{i\in I_r}K_{i,n}\xi_n^{(i)}.
\end{aligned}\]
\end{proposition}

\begin{corollary}[Exponential suppression of the cross-cluster correction]
\label{cor:cluster_mean_out_correction}
The cross-cluster correction in
Proposition~\ref{prop:exact_peki_cluster_mean_decomposition} satisfies
\[
    \left\|\delta_{r,n}^{\mathrm{out}}\right\|
    \leq
    \frac{B_n}{J_r}
    \sum_{i\in I_r}
    \frac{\varepsilon_{i,r,n}}
    {(1-\varepsilon_{i,r,n})^2D_{i,r,n}}
    \left\|y-A\theta_n^{(i)}\right\|.
    \label{eq:cluster_mean_out_bound}
\]
Consequently, by Proposition~\ref{prop:cross_cluster_weight_suppression}, the cross-cluster correction is exponentially suppressed as the cluster separation increases relative to \(\kappa\), provided that \(B_n\), the particle residuals, and \(D_{i,r,n}^{-1}\) remain bounded.
\end{corollary}

\begin{proposition}[Bound on the within-cluster correction]
\label{proposition:within_cluster}
Let $d_{r,n}^{\mathrm{in}}$ denote the within-cluster diameter, defined by $d_{r,n}^{\mathrm{in}}:= \max_{i,j\in I_r}d_{ij,n}$. Then the within-cluster correction in Proposition~\ref{prop:exact_peki_cluster_mean_decomposition} satisfies
\[
    \left\|\delta_{r,n}^{\mathrm{in}}\right\| \leq \frac{B_n}{2J_r} \tanh{\left(\frac{(d_{r,n}^{\mathrm{in}})^2}{4\kappa^2}\right)}\sum_{i\in I_r}\left\|y-A\theta_n^{(i)}\right\|.
\]
\end{proposition}

\begin{remark}[Polarization radius trade-off and scope]
\label{rem:peki_bandwidth_tradeoff}
The cross-cluster correction $\delta_{r,n}^{\mathrm{out}}$ captures the effect of particles outside $I_r$ on the cluster-mean update. The within-cluster correction $\delta_{r,n}^{\mathrm{in}}$ captures the effect of differences between the cluster-restricted PEKI gains and the ordinary cluster-EKI gain. The term \(\nu_{r,n}\) is due to observation perturbations. Decreasing \(\kappa\) can suppress sufficiently distant cross-cluster interactions and hence reduce \(\delta_{r,n}^{\mathrm{out}}\), but overly small \(\kappa\) can make the within-cluster gains more heterogeneous and may enlarge \(\delta_{r,n}^{\mathrm{in}}\). Thus, effective polarization requires $\kappa$ to balance within-cluster coupling against cross-cluster separation: $d_{r,n}^{\mathrm{in}} \ll\kappa\ll\min_{i\in I_r}d_{i,r,n}^{\mathrm{out}}$. No monotone dependence on \(\kappa\), cluster-mean convergence, or parameter-space convergence rate is claimed without further assumptions.

\end{remark}

\section{Hyperparameter Tuning for Polarization Radius}\label{sec: choose}

The polarization radius $\kappa$ defines the effective interaction range of PEKI: it sets the distance scale over which particles substantially influence a target particle's update, and hence the separation scale at which separated particle clusters become effectively decoupled.

We choose the polarization radius $\kappa$ by first identifying a candidate interval and then selecting a value using a one-step diagnostic scan, which requires no additional forward-model evaluations. A rough upper guideline is based on the largest relevant separation within the initial ensemble, while a rough lower guideline is set by the need to retain enough effective nearby particles for stable local covariance estimation. The final choice depends strongly on the forward-model outputs evaluated at the initial ensemble and therefore cannot be prescribed a priori; this dependence is assessed inexpensively through the one-step scan.

\subsection{Upper guideline} We first choose the upper guideline for \(\kappa\) from the prior-covered region or an expected separation between modes. In the discrete algorithm, the largest relevant interaction distance is set by the largest separation between particles present in the initial ensemble. The spatial extent of the prior-covered region therefore provides a natural upper guideline.

Consider the raw weight of a particle at distance \(r\) relative to the raw weight at zero distance:
\[
    \frac{\widetilde w(r)}{\widetilde w(0)}
    =
    \exp\left(-\frac{r^2}{2\kappa^2}\right).
\]
so a particle at \(r=3\kappa\) has only about \(1.1\%\) of the raw weight of a coincident particle. We therefore use \(3\kappa\) as an approximate effective interaction range. If \(D\) denotes the spatial range of the prior-covered region or the expected separation between modes, this gives the heuristic upper guideline $\kappa \lesssim D/3$ for reducing cross-cluster interaction. 

For example, when \(M_n=I\), a one-dimensional uniform prior on \([a,b]\) gives \(D=b-a\) and hence \(\kappa_u\lesssim(b-a)/3\). For a Gaussian prior with standard deviation \(\sigma\), taking the effective support as \([\mu-3\sigma,\mu+3\sigma]\) gives \(D=6\sigma\) and \(\kappa_u\lesssim2\sigma\). In higher dimensions, \(D\) should characterize the relevant separation directions rather than the full corner-to-corner diameter of the prior box.

Proposition~\ref{prop:cross_cluster_weight_suppression} also provides a theoretical upper guideline for \(\kappa\). By prescribing a tolerance on cross-cluster leakage, the exponential bound can be inverted to obtain a sufficient upper bound on \(\kappa\). This theoretical guideline is optional since the simpler heuristic guideline above is sufficient for the practical procedure.

\subsection{Lower guideline} Theoretically, the smallest practically useful value of \(\kappa\) is limited by the initial ensemble, since stable estimation of the local covariance matrices requires sufficiently many neighboring particles within the effective interaction range. In principle, a lower guideline can therefore be obtained by imposing a minimum effective sample size $\mathrm{ESS}_{i,n}(\kappa) =\left[\sum_{j=1}^J w_{ij,n}(\kappa)^2\right]^{-1}$ to exceed a prescribed threshold. In practice, however, this lower guideline need not be determined precisely; it only needs to be small enough for the initial candidate set to reach a strongly localized regime. We recommend \(\kappa_l=0.05D\) as a simple initial choice relative to the reference scale \(D\). If the scan does not reach such a regime, the candidate interval can be extended toward smaller \(\kappa\). Because all candidates reuse the cached forward-model outputs and require only a single PEKI update, this adjustment requires no additional forward-model evaluations and is typically inexpensive.


\subsection{One-step scan} Finally, we perform an efficient one-step diagnostic scan over the candidate set to select \(\kappa\) from the empirical effect of a single PEKI update. For each candidate value of \(\kappa\), we reuse the same initial ensemble, cached forward-model outputs, and pairwise distances; only the kernel weights, local means and covariances, Kalman gains, and particle updates are recomputed. Therefore, if the ensemble size is \(J\), the scan requires only \(J\) forward-model evaluations, independently of the number of candidate values of \(\kappa\). A single PEKI step is not intended to fully resolve the final solutions but provides a trend diagnostic. As \(\kappa\) decreases, the tendency toward mode separation typically becomes stronger, reaches a peak, and then weakens; see \cref{fig:himmelblau_one_step_peki_kappa_scan}. We select \(\kappa_\star\) from this intermediate regime, guided visually and by larger density-based clustering validation (DBCV) values. If this trend does not appear, the candidate set is extended toward smaller \(\kappa\); if it is present but not clearly resolved, the grid is refined.

After $\kappa_\star$ is selected, the corresponding updated ensemble is not retained as the initial ensemble for the remaining PEKI iterations. Thus, the diagnostic update also serves as the first step of the full \(N\)-step PEKI run. Scanning additional candidate values adds only the weighting, covariance, and update computations associated with the unselected candidates.

\begin{algorithm}
\caption{Practical selection of the polarization radius $\kappa$}
\label{alg:kappa_selection}
\begin{algorithmic}

\STATE Draw the initial ensemble from prior $\{\theta_0^{(j)}\}_{j=1}^J$ .

\STATE Evaluate the forward map at the initial ensemble once $g_0^{(j)}=\mathcal{G}(\theta_0^{(j)})$ and cache the responses $\{g_0^{(j)}\}_{j=1}^J$.

\STATE Determine an upper guideline $\kappa_u$ based on the reference scale $D$ or Proposition~\ref{prop:cross_cluster_weight_suppression}.

\STATE Determine a candidate set $\{\kappa_l,\kappa_1,\kappa_2, \cdots, \kappa_u\}, \quad 0 < \kappa_l < \kappa_1 \cdots < \kappa_u$.\

\STATE Scan the candidate set as follows:
\FOR{each candidate $\kappa$}
    \STATE Run one step of PEKI update for $\{\theta_1^{(j)}(\kappa)\}_{j=1}^J$  via Eq. \cref{eq: peki update}.
\ENDFOR

\STATE Select $\kappa_\star$ in the intermediate regime showing clear mode separation while retaining appreciable particle motion. Large DBCV values are also preferred.

\IF{no suitable candidate is identified}
    \STATE Extend the candidate set toward smaller values of $\kappa$,
    or refine the grid around the observed transition regime, and repeat the scan.
\ENDIF

\RETURN $\kappa_\star$ and $\{\theta_1^{(j)}(\kappa_*)\}_{j=1}^J$.

\end{algorithmic}
\end{algorithm}

\subsection{Computational cost}

When considering the computational requirements of PEKI itself, excluding the one-step procedure used to select \(\kappa\), the main considerations are the same as ~\cite{wacker2026perspectives}. In contrast to standard EKI, PEKI requires particle-dependent weighted moments for every particle. The main additional computational effort therefore lies in computing these weighted covariances. With a naive implementation, the covariance associated with each of the $J$ particles requires a weighted average over all $J$ ensemble members, resulting in $O(J^2)$ computational cost per iteration. This additional overhead does not require additional forward-model evaluations.

The computational cost of the one-step diagnostic scan is considered separately. The scan requires no additional forward-model evaluations, since all candidate values of \(\kappa\) are evaluated using the forward-model outputs already computed for the same ensemble. The remaining computational effort consists of computing the pairwise particle distances and, for each candidate \(\kappa\), the corresponding weights, local means and covariances, Kalman gains, and candidate updates. The pairwise distances need to be computed only once and can be reused for all candidate values. Thus, for $K$ candidate values, the naive computational cost of the scan scales as $O(KJ^2)$.

\section{Numerical results}
\label{sec: Numerical Results}

The numerical experiments examine the practical behavior of PEKI in finite-ensemble, discrete-time inverse problems, with emphasis on the identification of multiple data-consistent solutions across problems of increasing complexity, practical selection of the polarization radius, and the mechanisms predicted by the theoretical analysis.

The key motivations and conclusions of the numerical results are summarized below:

\textbullet\ We first use a two-dimensional Himmelblau inverse problem to demonstrate practical selection of the polarization radius through the proposed one-step diagnostic scan. This example also provides empirical evidence that the theoretical cross-cluster suppression bound is sharp at the exponential scale and shows the empirical late-stage decay of the cluster-wise MSE.

\textbullet\ We then consider a Rastrigin-style inverse problem with a highly nonconvex likelihood landscape containing several dominant and secondary modes, and compare PEKI with standard stochastic EKI, GMM-EKI, and GMKI. In the reported runs, PEKI maintains separated clusters near all dominant modes and multiple secondary modes, whereas the baselines recover only part of the multimodal structure.

\textbullet\ Finally, we test PEKI on a high-dimensional Navier--Stokes inverse problem with a bimodal posterior. Across \(K=10\), \(50\), and \(128\) dimensional cases, PEKI identifies the true coefficient values in most dimensions and reconstructs visually accurate vorticity fields from the identified modes, with relative reconstruction errors below \(10\%\).

\subsection{Himmelblau inverse problem}

We consider a two-dimensional Himmelblau inverse problem ~\cite{bungert2025polarized, wacker2026perspectives} to demonstrate practical selection of the polarization radius $\kappa$ through the proposed one-step scan. For \(\theta=(\theta_1,\theta_2)^\top \in \mathbb{R}^2\), the nonlinear forward map is
\[
    G(\theta)
=
\begin{pmatrix}
\theta_1^2+\theta_2 \\
\theta_1+\theta_2^2
\end{pmatrix} + \eta,
\qquad
\eta \sim \mathcal{N}(0,\Gamma), \quad \Gamma=0.5^2I.
\]

Given the observation $y=(11,7)^\top$, the corresponding data-misfit function is
\[
    \Phi(\theta)
=
\frac{1}{2}
\left\|
\Gamma^{-1/2}\bigl(G(\theta)-y\bigr)
\right\|^2
=
\frac{1}{2}
\left[
\left(\frac{\theta_1^2+\theta_2-11}{0.5}\right)^2
+
\left(\frac{\theta_1+\theta_2^2-7}{0.5}\right)^2
\right],
\]
with four zero-misfit solutions, corresponding to the four posterior modes under the uniform prior. We initialize $J=1500$ particles from the uniform prior on $[-6,6]^2$ and run PEKI for $20$ iterations.

Figure~\ref{fig:himmelblau_results} compares the four-mode reference posterior with standard EKI and PEKI. Figure~\ref{fig:himmelblau_standard_eki} shows the final standard stochastic EKI ensembles: the update moves particles toward the low-misfit region but does not separate them into four distinct clusters. The one-step scan in Fig.~\ref{fig:himmelblau_one_step_peki_kappa_scan} shows that large $\kappa$ produces an update close to standard EKI. As $\kappa$ decreases, distinct clusters emerge and remain stable near $\kappa=1.5$, with the DBCV value reaching its maximum of $0.2811$. For still smaller $\kappa$, the one-step particle displacements become small, and the ensemble changes little, consistent with the small-$\kappa$ behavior described in Remark~\ref{rem:small_kappa_numerical_implications}. We therefore select $\kappa=1.5$. Over the full $20$ iterations, the particles gradually concentrate around all four zero-misfit solutions (Fig.~\ref{fig:himmelblau_full_peki_kappa1p5}), demonstrating that the one-step scan can select the polarization radius before the full PEKI run.

\begin{figure}[!htbp]
    \centering
    \begin{subfigure}[t]{0.22\textwidth}
        \centering
        \includegraphics[width=\textwidth]{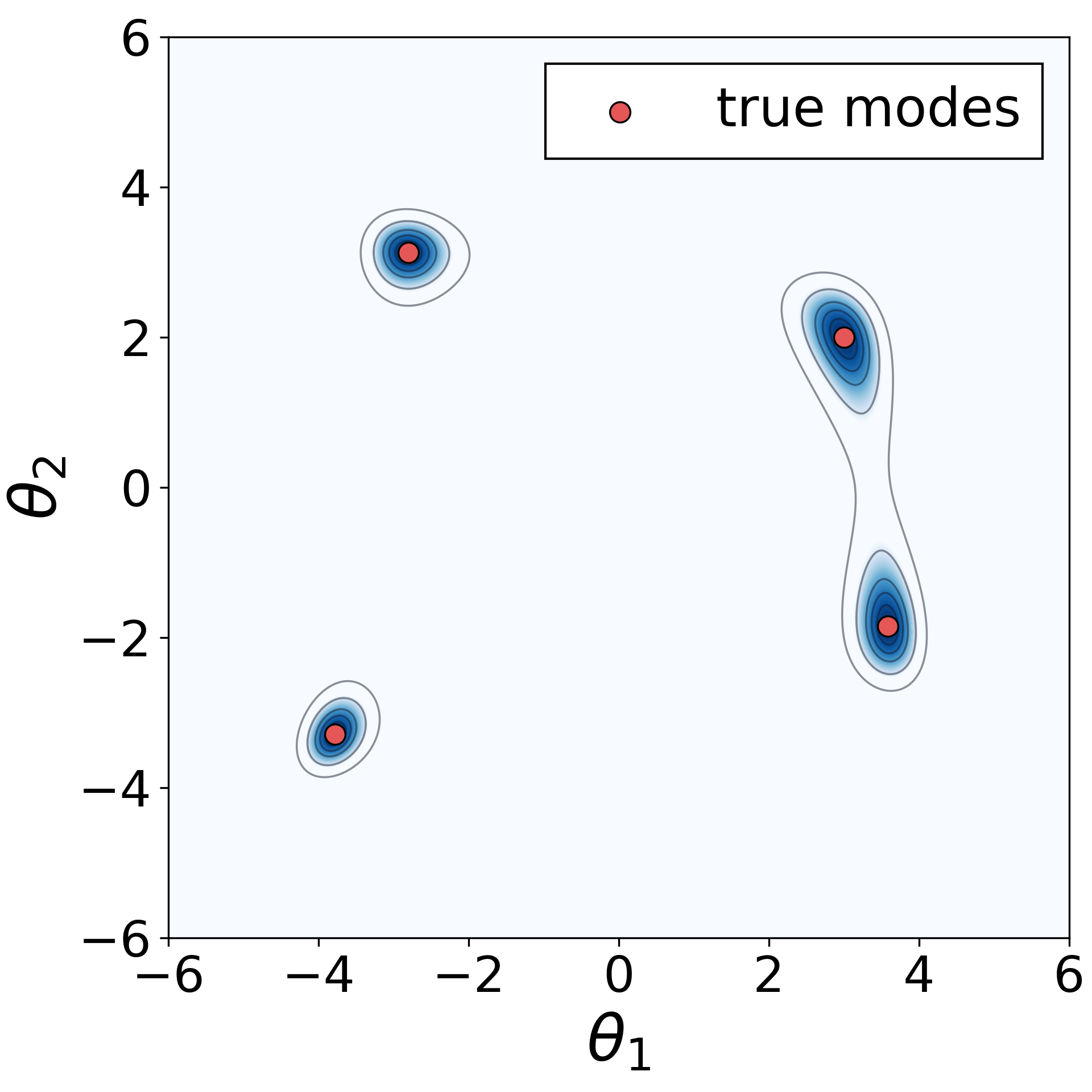}
        \caption{Reference posterior.}
        \label{fig:himmelblau_reference_posterior}
    \end{subfigure}
    \hfill
    \begin{subfigure}[t]{0.22\textwidth}
        \centering
        \includegraphics[width=\textwidth]{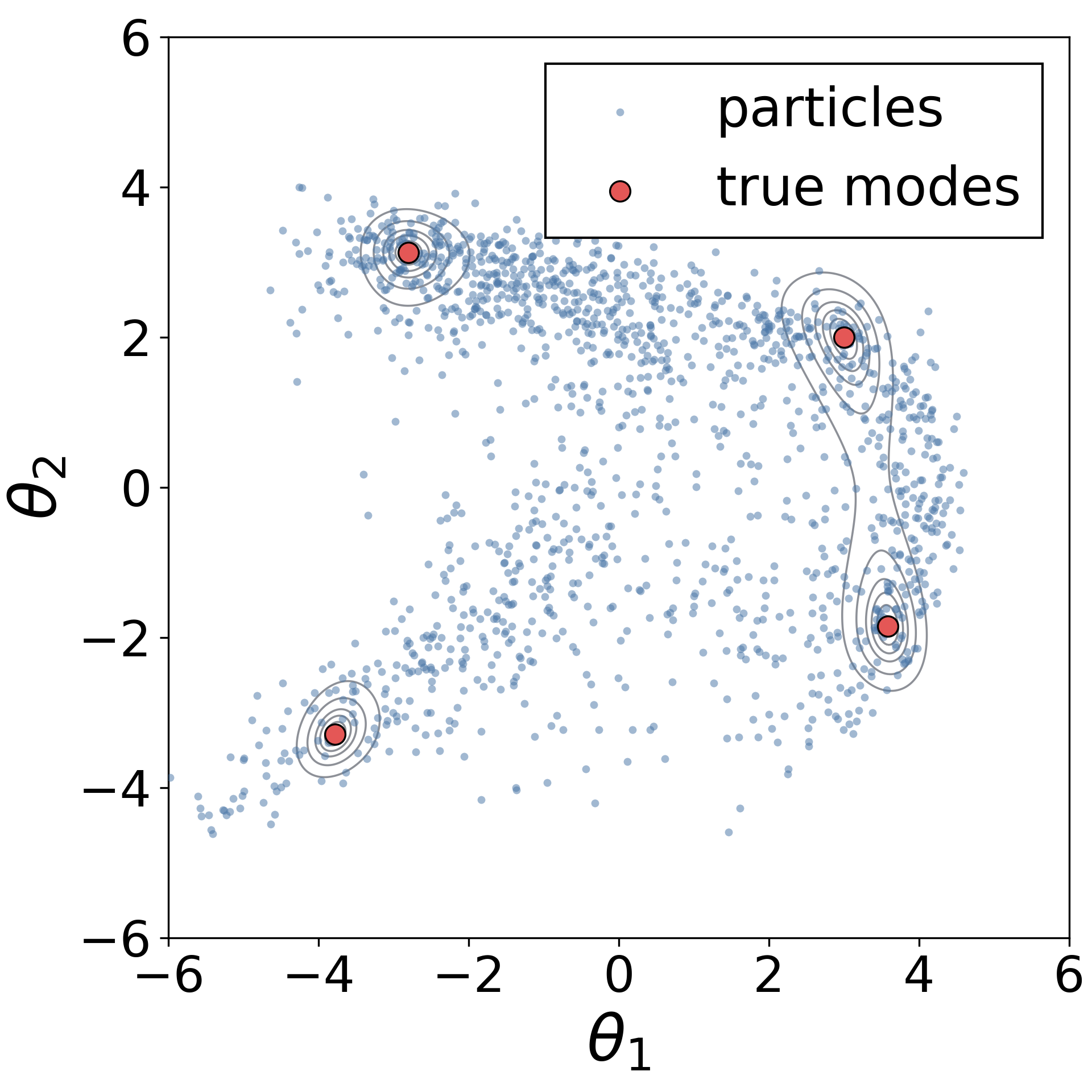}
        \caption{Standard EKI.}
        \label{fig:himmelblau_standard_eki}
    \end{subfigure}
    \hfill
    \begin{subfigure}[t]{0.46\textwidth}
        \centering
        \includegraphics[width=\textwidth]{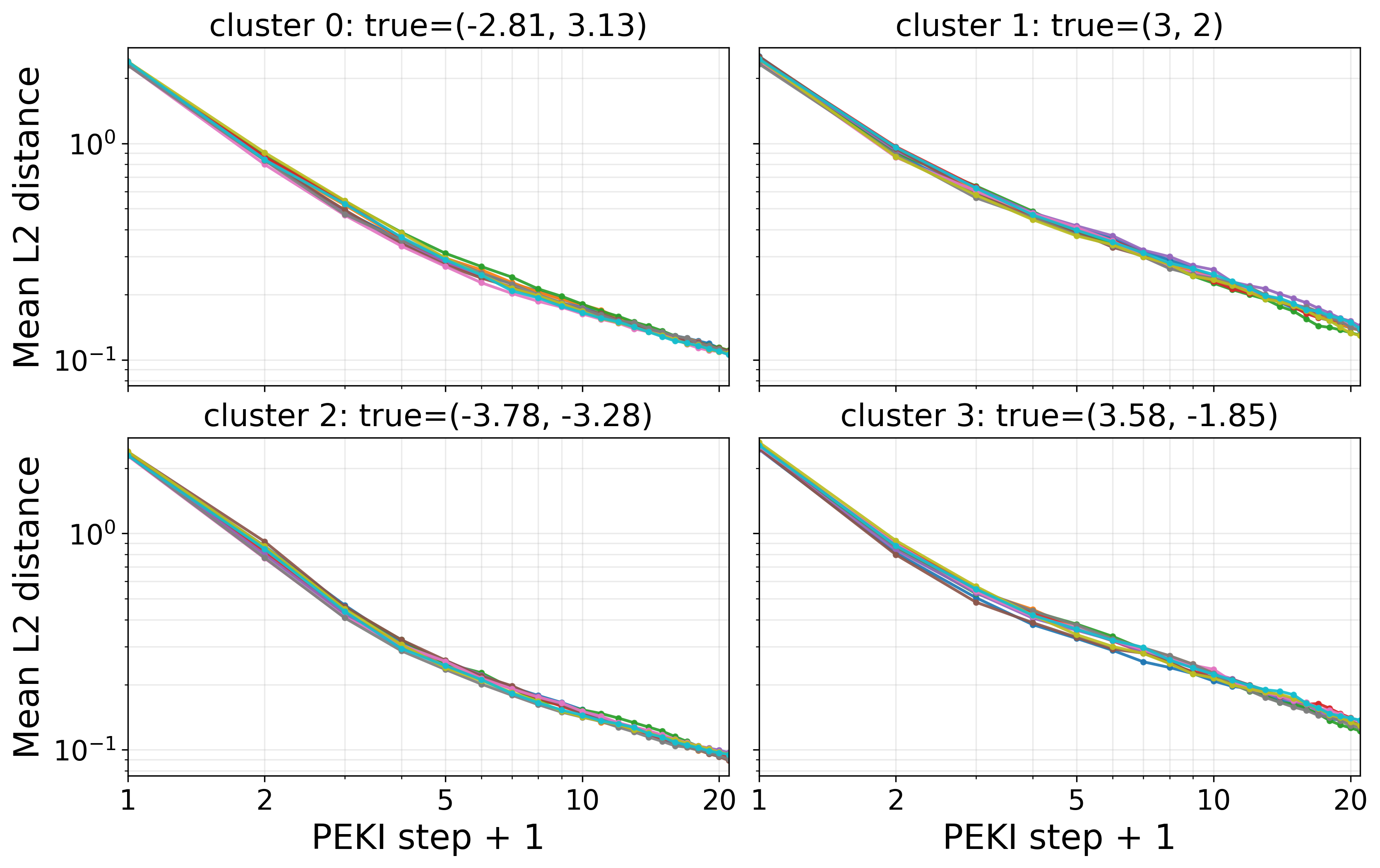}
        \caption{Cluster-wise MSE.}
        \label{fig:himmelblau_peki_convergence}
    \end{subfigure}
    
    \vspace{0.5em}
    
    \begin{subfigure}[t]{0.48\textwidth}
        \centering
        \includegraphics[width=\textwidth]{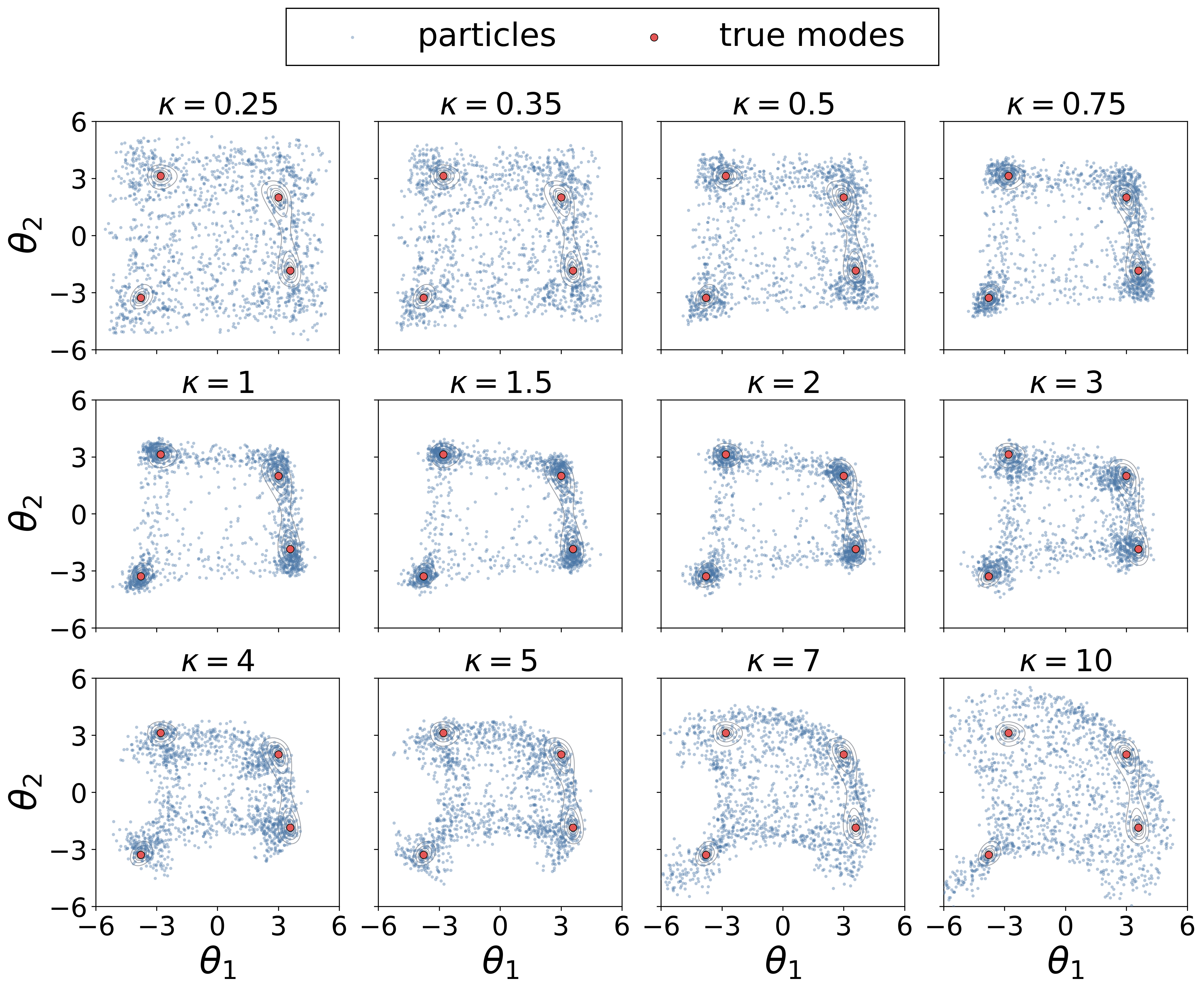}
        \caption{One-step PEKI scan over \(\kappa\).}
        \label{fig:himmelblau_one_step_peki_kappa_scan}
    \end{subfigure}
    \hfill
    \begin{subfigure}[t]{0.48\textwidth}
        \centering
        \includegraphics[width=\textwidth]{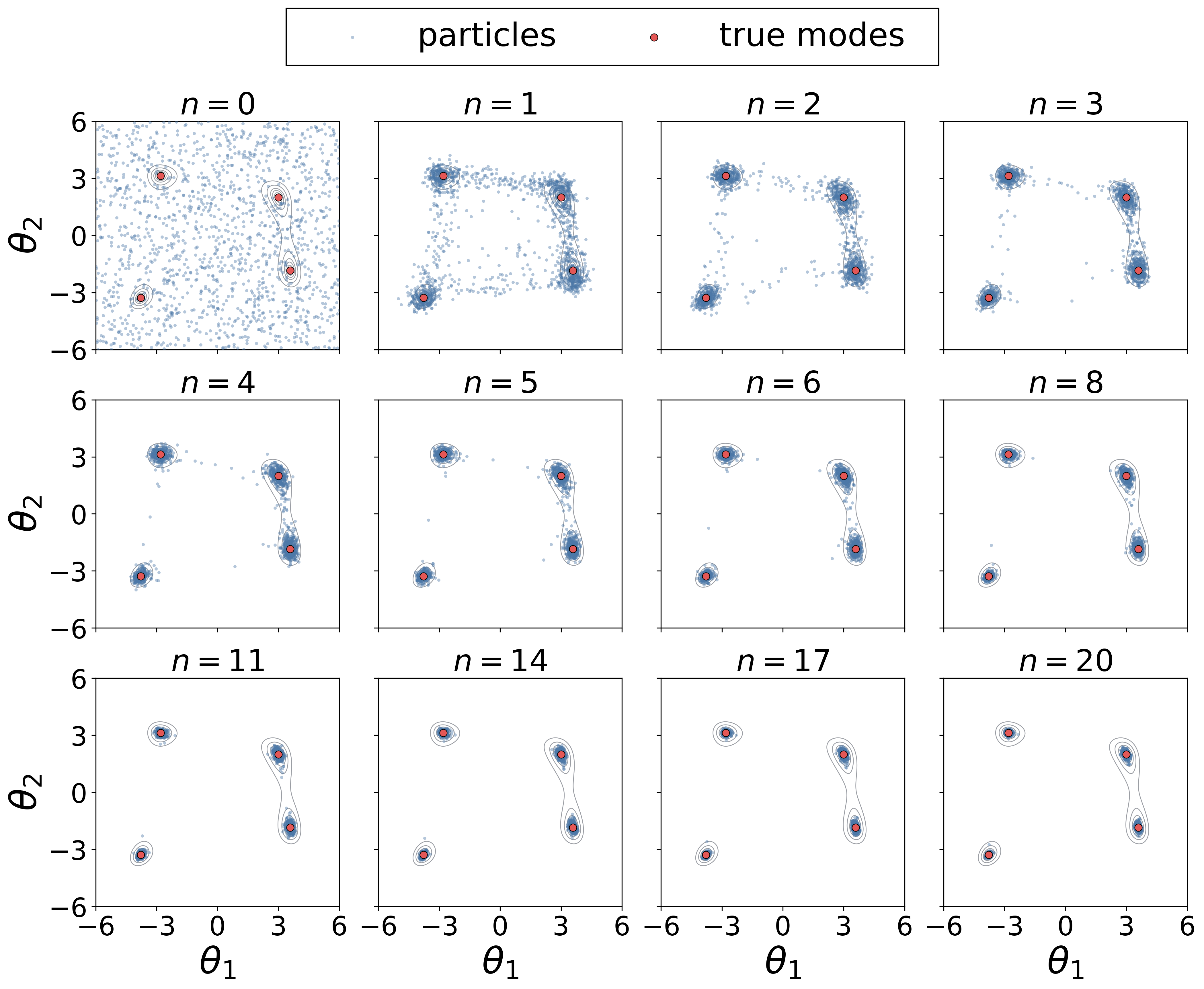}
        \caption{Full PEKI with \(\kappa=1.5\).}
        \label{fig:himmelblau_full_peki_kappa1p5}
    \end{subfigure}
    \caption{Himmelblau inverse problem. Panels (a)--(e) show the reference posterior, the final standard stochastic EKI ensembles, cluster-wise MSE on a log--log scale, the one-step scan over \(\kappa\), and the full PEKI result with \(\kappa=1.5\), respectively.}
    \label{fig:himmelblau_results}
\end{figure}

To examine the exponential suppression of cross-cluster interactions, we fix the ensemble and vary the polarization radius $\kappa$. For each particle $i\in I_r$, we evaluate the cross-cluster leakage $\varepsilon_{i,r,n}(\kappa)$ from Proposition~\ref{prop:cross_cluster_weight_suppression}, together with the discrepancy between the full PEKI update and the cluster-restricted PEKI update, $\left\| \theta_{n+1}^{(i)}(\kappa) - \widetilde{\theta}_{n+1}^{(i,r)}(\kappa) \right\|_2$ from Corollary~\ref{cor:cluster_update_discrepancy}.

The observed cross-cluster leakage and update discrepancy closely follow the exponential dependence on $1/\kappa^2$ predicted by the theoretical bounds. Figure~\ref{fig:himmelblau_exponential_cross_cluster} plots both quantities against $1/\kappa^2$ on logarithmic vertical scales. The leakage (Fig.~\ref{fig:himmelblau_weight_leakage}) and update discrepancy (Fig.~\ref{fig:himmelblau_update_discrepancy}) are approximately linear on the semilogarithmic scale before the latter reaches numerical precision, indicating that the leading exponential scale already dominates over the resolved range and is consistent with Proposition~\ref{prop:cross_cluster_weight_suppression} and Corollary~\ref{cor:cluster_update_discrepancy}. Zero discrepancies are displayed using a floor of $2.78\times10^{-18}$, one tenth of the smallest nonzero value, which creates the apparent plateau at small $\kappa$.

To quantify the update-discrepancy exponential trend in Fig.~\ref{fig:himmelblau_update_discrepancy}, we fit the discrepancy for each particle by
\[
\left\| \theta_{n+1}^{(i)}(\kappa) - \widetilde{\theta}_{n+1}^{(i,r)}(\kappa) \right\|_2
\approx
A_i
\exp\left(-\frac{\lambda_i}{\kappa^2}\right),
\]
where $A_i$ and $\lambda_i$ are particle-dependent fitting parameters. To avoid contamination from the numerical-precision regime, the fit is restricted to discrepancy values above a numerical cutoff of $10^{-14}$, corresponding approximately to $1\le 1/\kappa^2<8.2$. Among the particles, the median particle-wise $R^2$ is $0.99977$, and the $10$th percentile is $0.99748$, confirming near-linearity on the semilogarithmic scale.

Corollary~\ref{cor:cluster_update_discrepancy} suggests a leading decay rate $(d_{i,r,n}^{\rm out})^2/2$, whereas the correction term decays at the faster rate $(d_{i,r,n}^{\rm out})^2$. We accordingly compare the fitted particle-wise rates $\lambda_i$ with $(d_{i,r,n}^{\rm out})^2/2$. After estimating $\lambda_i$ separately for each particle, we regress $\lambda_i$ against $(d_{i,r,n}^{\rm out})^2/2$ across all particles and obtain
\[
\lambda_i
\approx
0.055
+
1.0007\,\frac{(d_{i,r,n}^{\rm out})^2}{2},
\qquad
R^2 = 0.9949.
\]
Thus, the discrepancy follows the predicted exponential dependence on $1/\kappa^2$, and its fitted particle-wise rates closely match the leading rates $(d_{i,r,n}^{\rm out})^2/2$. These results provide empirical evidence that this leading exponential scale is attained by the actual update discrepancy and is therefore sharp at the level of the exponential rate.

We further examine the empirical convergence behavior. PEKI recovers separated particle clusters near all modes of the reference solution as
\begin{equation}
    \text{MSE}=\frac{1}{|I_r|} \sum_{i\in I_r} \left\|\theta_n^{(i)}-\theta_r\right\|_2^2.
    \label{eq:mse error}
\end{equation}
Figure~\ref{fig:himmelblau_peki_convergence} reports the cluster-wise MSE over 10 random seeds. The log--log curves decrease rapidly during the first three iterations and more slowly thereafter. Fitting the remaining 18 values with a power law \(e_n=Cn^{-\alpha}\) gives \(R^2\in[0.9950,0.9996]\) across the four clusters, while the floor-corrected geometric model \(e_n=e_\infty+Cq^n\) gives \(R^2\in[0.9908,0.9967]\). Both models therefore describe the late-stage decay well.

\begin{figure}[!htbp]
    \centering
    \begin{subfigure}[t]{0.48\textwidth}
        \centering
        \includegraphics[width=\textwidth]{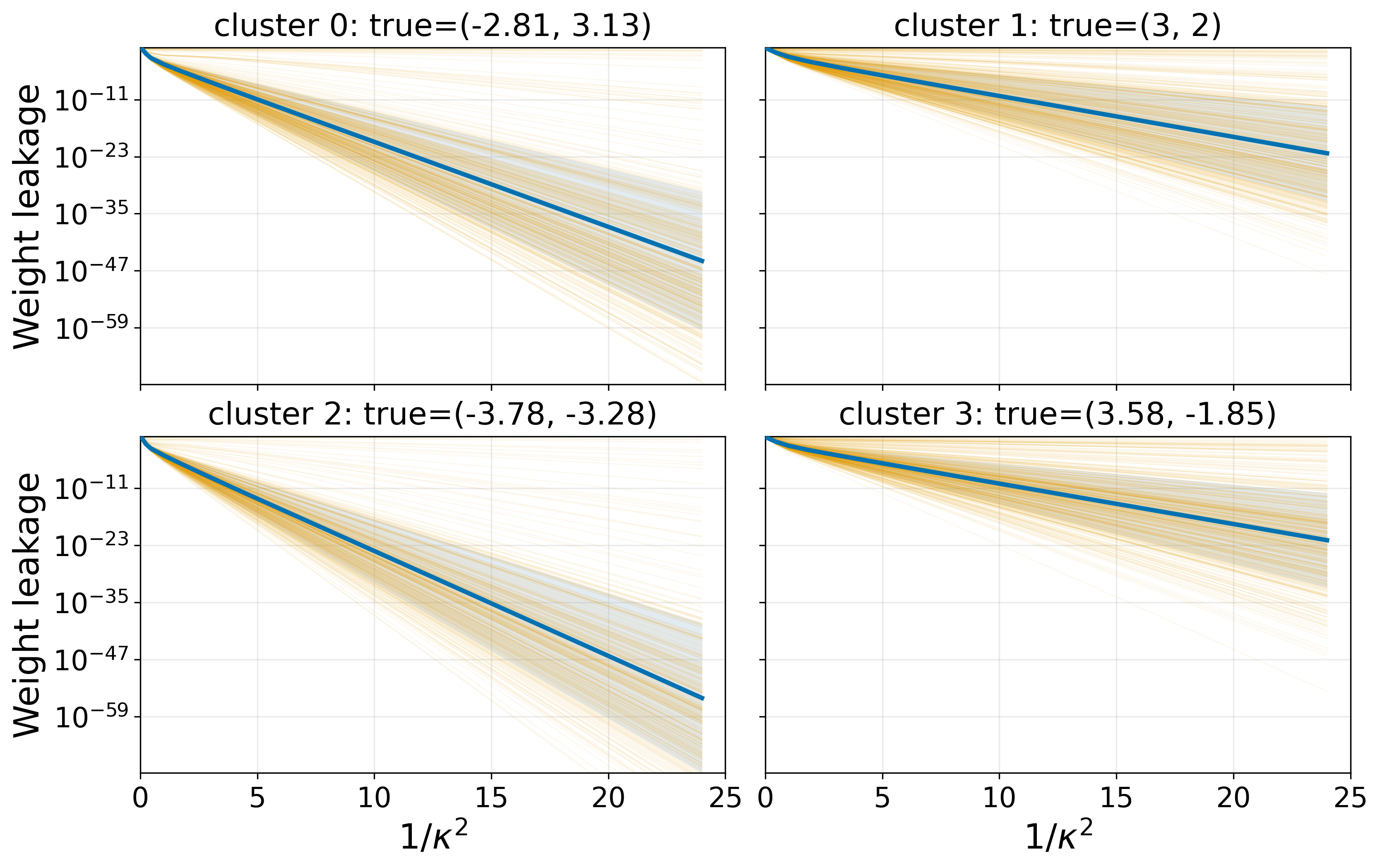}
        \caption{Cross-cluster weight leakage.}
        \label{fig:himmelblau_weight_leakage}
    \end{subfigure}
    \hfill
    \begin{subfigure}[t]{0.48\textwidth}
        \centering
        \includegraphics[width=\textwidth]{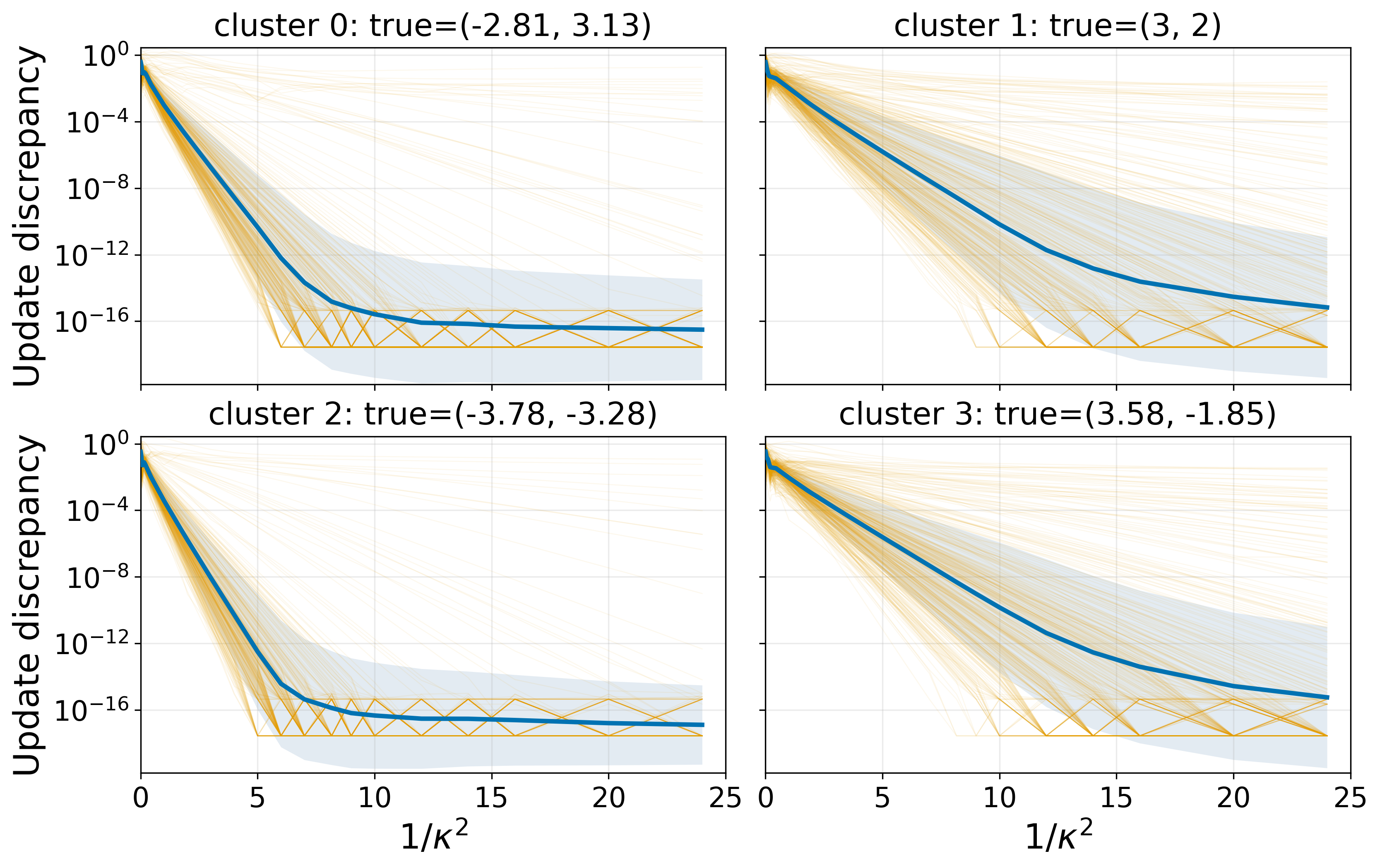}
        \caption{Update discrepancy.}
        \label{fig:himmelblau_update_discrepancy}
    \end{subfigure}
    \caption{Empirical cross-cluster suppression consistent with the theoretical exponential upper bound. Panels (a) and (b) show the leakage and the full-versus-cluster-restricted update discrepancy against $1/\kappa^2$ on semilogarithmic scales. Both exhibit approximately linear semilogarithmic trends, and the update discrepancy eventually reaches numerical precision. Thin orange curves represent individual particles, the thick blue curve the ensemble mean, and the gray band one standard deviation.}
    \label{fig:himmelblau_exponential_cross_cluster}
\end{figure}

\subsection{Rastrigin inverse problem.}

Inspired by the Rastrigin-type function in \cite{bungert2025polarized}, we construct a Rastrigin-style inverse problem to test whether PEKI maintains coverage of multiple solution modes in a highly nonconvex landscape with many competing local minima.

For $\theta=(\theta_1,\theta_2)^\top\in\mathbb{R}^2$, the two-dimensional Rastrigin function \cite{rastrigin1974systems} is
\[
    R(z)=20+\sum_{i=1}^2\left(z_i^2-10\cos(2\pi z_i)\right),
\]
and the three-mode potential in \cite{bungert2025polarized} multiplies shifted copies of this oscillatory landscape:\[
    V(\theta)
=
\frac{1}{8}
R(\theta)\,
R\!\left(\theta-\begin{bmatrix}3\\2\end{bmatrix}\right)\,
R\!\left(\theta+\begin{bmatrix}1\\3.5\end{bmatrix}\right).
\]

We define an inverse problem with a scalar forward map
\[
G(\theta)=\sqrt{2\max\{V(\theta),0\}},
\]
and an artificial observation $y=0$ under zero-mean Gaussian noise with variance $\Gamma=1$. The resulting data-misfit is
\[
    \Phi(\theta)
=
\frac{1}{2}\left(G(\theta)-y\right)^\top
\Gamma^{-1}
\left(G(\theta)-y\right)
=
\frac{1}{2}G(\theta)^2.
\]

A local minimum of $V$ is a local maximum of the log-posterior because
\[
    \pi(\theta\mid y)
    =
    \exp \big[-V(\theta)+\mathrm{const.}\big]
\]
The exponential transformation makes the potential and posterior visualizations differ substantially. The three global minima of $V$, $(0,0)$, $(3,2)$, and $(-1,-3.5)$, appear as the dominant posterior modes in Fig.~\ref{fig:rastrigin_reference_posterior}. Other local minima below the threshold $V_{\mathrm{crit}} = -\log(\epsilon_{\mathrm{mach}})
    \approx 36.04$ are exponentially suppressed and disappear on the linear posterior-density scale while remaining visible in the log-posterior plot (Fig.~\ref{fig:rastrigin_reference_posterior_log}). Thirdly, minima above $V_{\mathrm{crit}}$ have relative posterior density below the double-precision machine epsilon $\epsilon_{\mathrm{mach}}\approx2.22\times10^{-16}$ and are clipped to zero, so their structures remain invisible even after taking the logarithm. The domain therefore contains three categories of structures: three dominant modes, secondary modes visible only on the log-posterior scale, and highly suppressed minima of $V$ that are numerically invisible in the posterior visualization.

We compare standard stochastic EKI, with and without covariance inflation, against PEKI under the same ensemble budget. Each method uses $J=15000$ particles, $N=20$ iterations, and step size $h=1/20$; PEKI uses $\kappa=0.1$.

PEKI maintains particle coverage across substantially more separated low-misfit regions than standard EKI and the tested Gaussian-mixture baselines in this highly oscillatory problem. Figure~\ref{fig:rastrigin_results} shows the final ensembles in the parameter plane and their coordinate-wise marginals. Standard stochastic EKI does not represent the multimodal posterior reliably: some basin structure appears in the $\theta_1$ marginal, whereas the $\theta_2$ marginal remains essentially Gaussian-like and misses the separated high-density regions (Fig.~\ref{fig:rastrigin_final_ensemble_standard_eki}). Covariance inflation reveals more basin structure in $\theta_2$ but degrades the $\theta_1$ marginal (Fig.~\ref{fig:rastrigin_final_ensemble_standard_eki_inflation}).

PEKI instead produces separated clusters near the three dominant posterior modes and near secondary modes associated with local minima of $V$ visible on the log-posterior scale (Fig.~\ref{fig:rastrigin_final_ensemble_peki}). It therefore maintains particle coverage across multiple low-misfit regions of the oscillatory landscape. PEKI also forms clusters near highly suppressed local basins of $V$ that are not visible as peaks even in the log-posterior plot. These should not be interpreted as recovered posterior components and are more likely caused by the interaction of the oscillatory landscape with stochastic EKI perturbations. Restricting the prior from $[-6,6]^2$ to $[-2,4]\times[-4,4]$ substantially reduces or eliminates these spurious clusters (Fig.~\ref{fig:rastrigin_final_ensemble_peki_small_prior}).

We show the results of other baseline methods: GMM-EKI~\cite{dovera2011multimodal} in Fig.~\ref{fig:rastrigin_gmm_eki} and GMKI~\cite{chen2024efficient} in Fig.~\ref{fig:rastrigin_gmki}.  In these runs, both Gaussian-mixture multimodal methods can degenerate into a few components and therefore have difficulty recovering all possible solutions in this challenging case. More extensive hyperparameter tuning or repeated runs may improve performance.

Overall, these results support the central purpose of the example: the polarization enables PEKI to maintain particle coverage across multiple low-misfit regions in a highly oscillatory landscape with many competing basins. At the same time, the naive PEKI implementation identifies multiple low-misfit regions but does not aim to represent the posterior distribution or recover their relative posterior masses.

\begin{figure}[!htbp]
    \centering
    \begin{subfigure}[t]{0.31\textwidth}
        \centering
        \includegraphics[width=\textwidth]{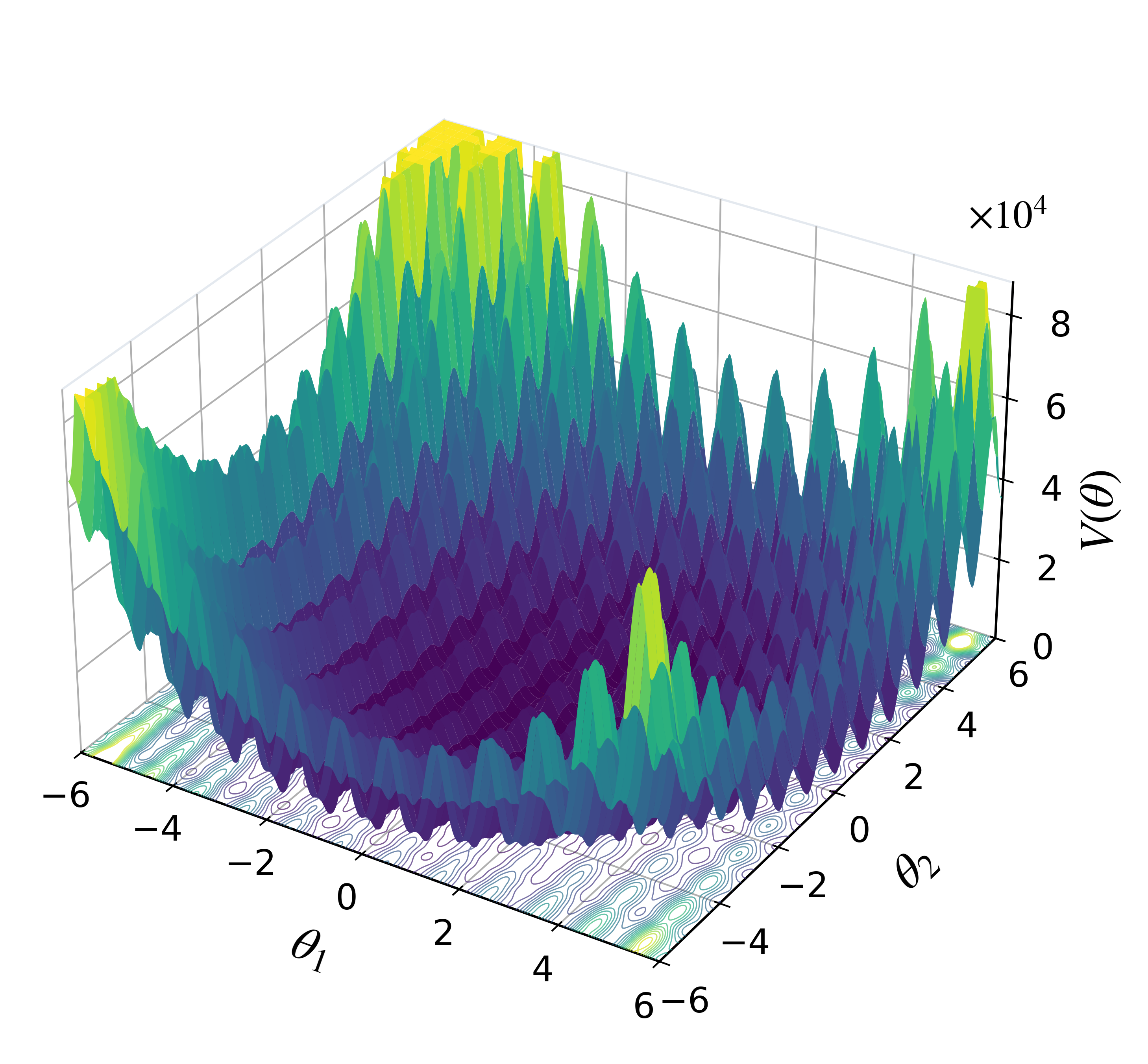}
        \caption{Potential landscape.}
        \label{fig:rastrigin_potential_landscape}
    \end{subfigure}
    \hfill
    \begin{subfigure}[t]{0.31\textwidth}
        \centering
        \includegraphics[width=\textwidth]{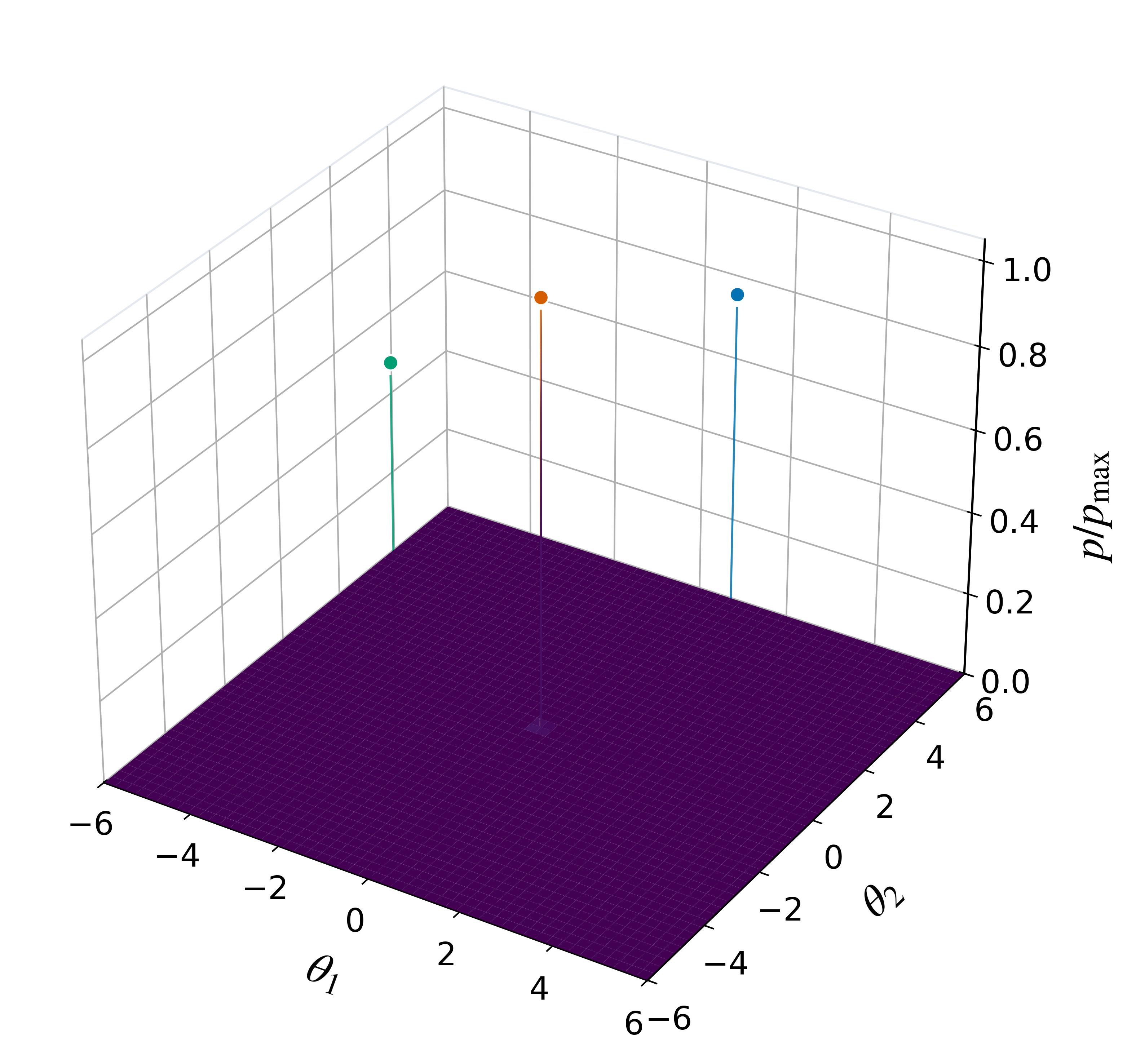}
        \caption{Reference posterior.}
        \label{fig:rastrigin_reference_posterior}
    \end{subfigure}
    \hfill
    \begin{subfigure}[t]{0.31\textwidth}
        \centering
        \includegraphics[width=\textwidth]{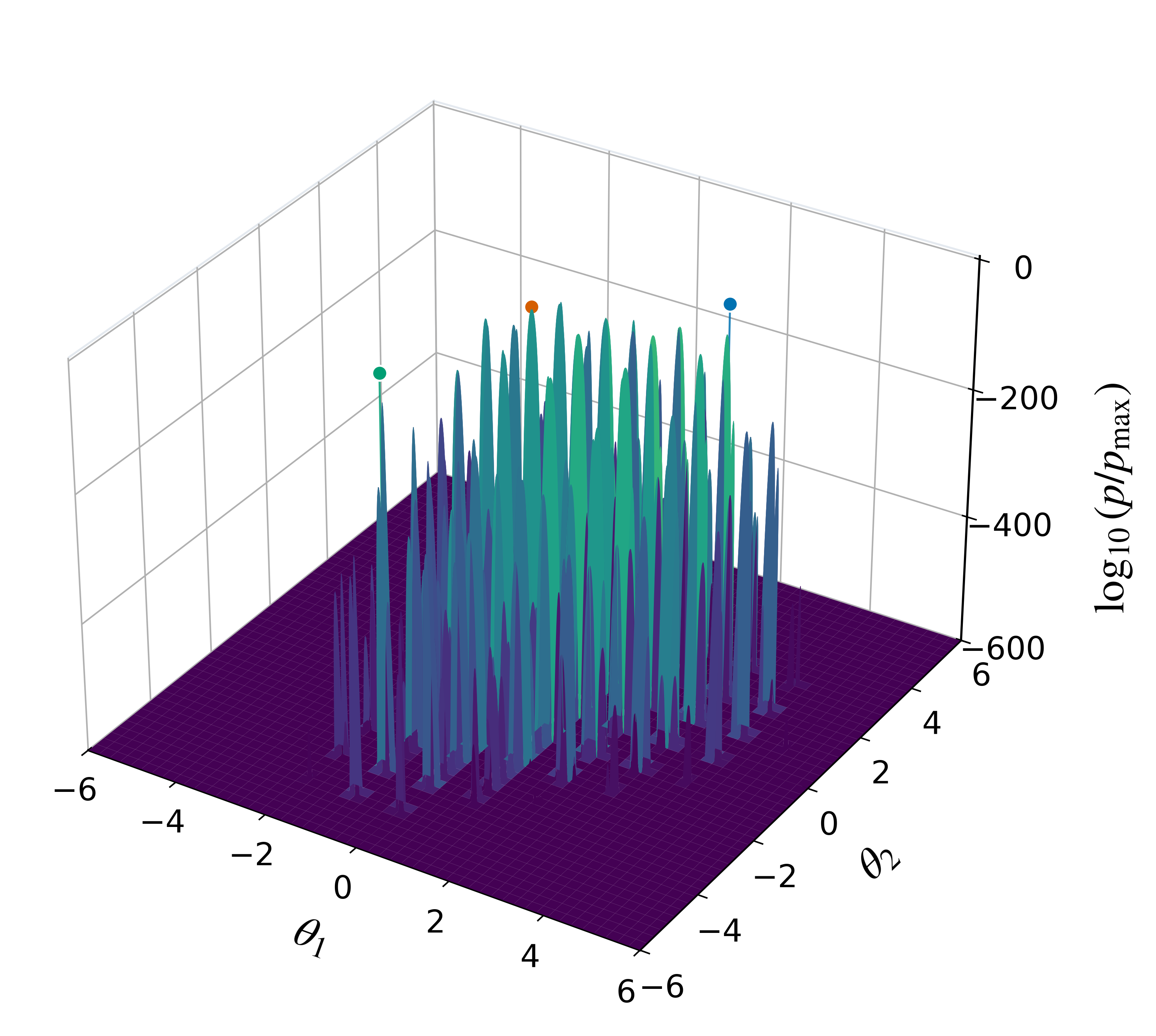}
        \caption{Log-posterior landscape.}
        \label{fig:rastrigin_reference_posterior_log}
    \end{subfigure}
    
    \vspace{0.5em}
    
    \begin{subfigure}[t]{0.32\textwidth}
        \centering
        \includegraphics[width=\textwidth]{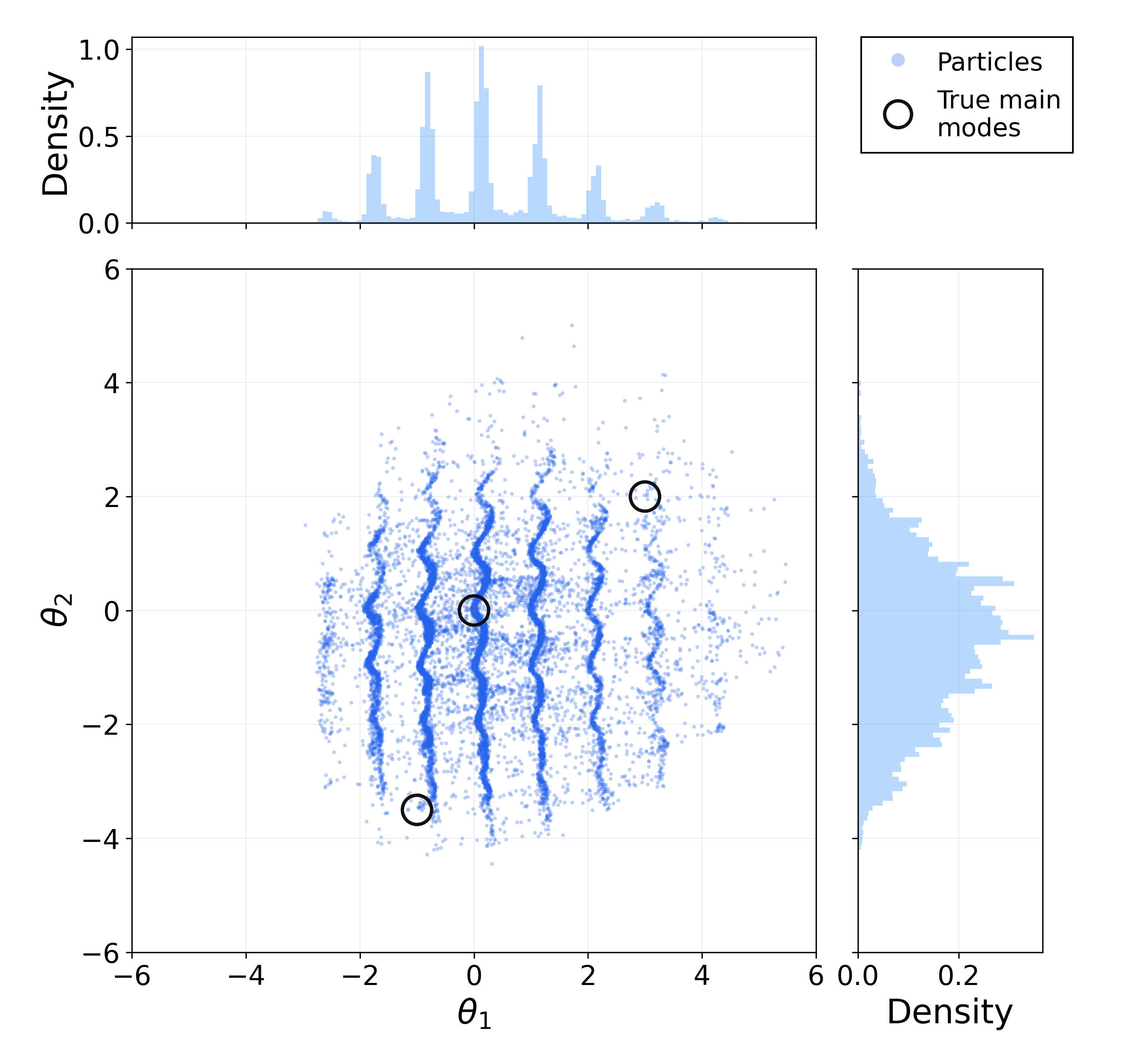}
        \caption{Standard EKI.}
        \label{fig:rastrigin_final_ensemble_standard_eki}
    \end{subfigure}
    \hfill
    \begin{subfigure}[t]{0.32\textwidth}
        \centering
        \includegraphics[width=\textwidth]{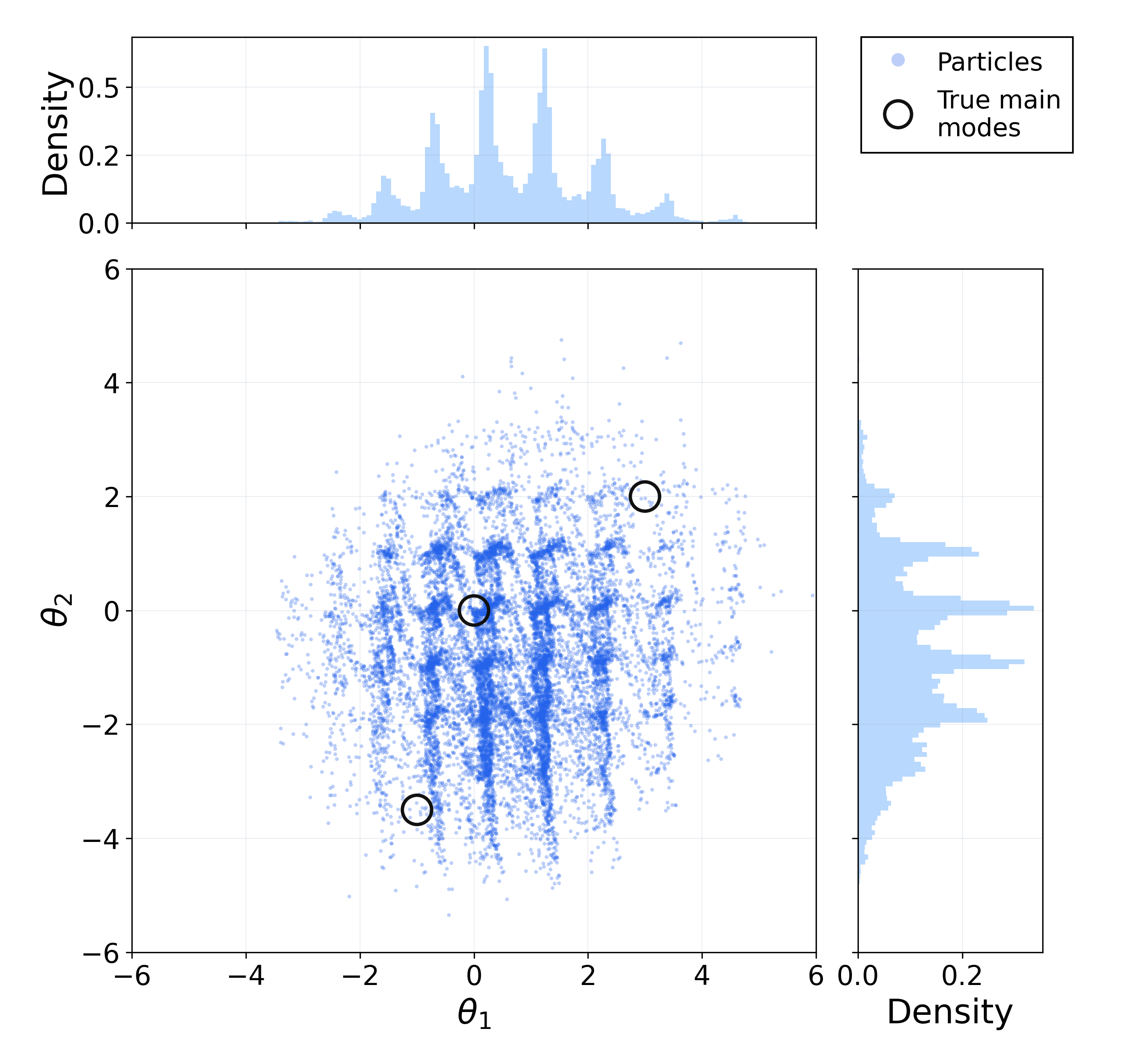}
        \caption{Standard EKI with inflation.}
        \label{fig:rastrigin_final_ensemble_standard_eki_inflation}
    \end{subfigure}
    \hfill
    \begin{subfigure}[t]{0.32\textwidth}
        \centering
        \includegraphics[width=\textwidth]{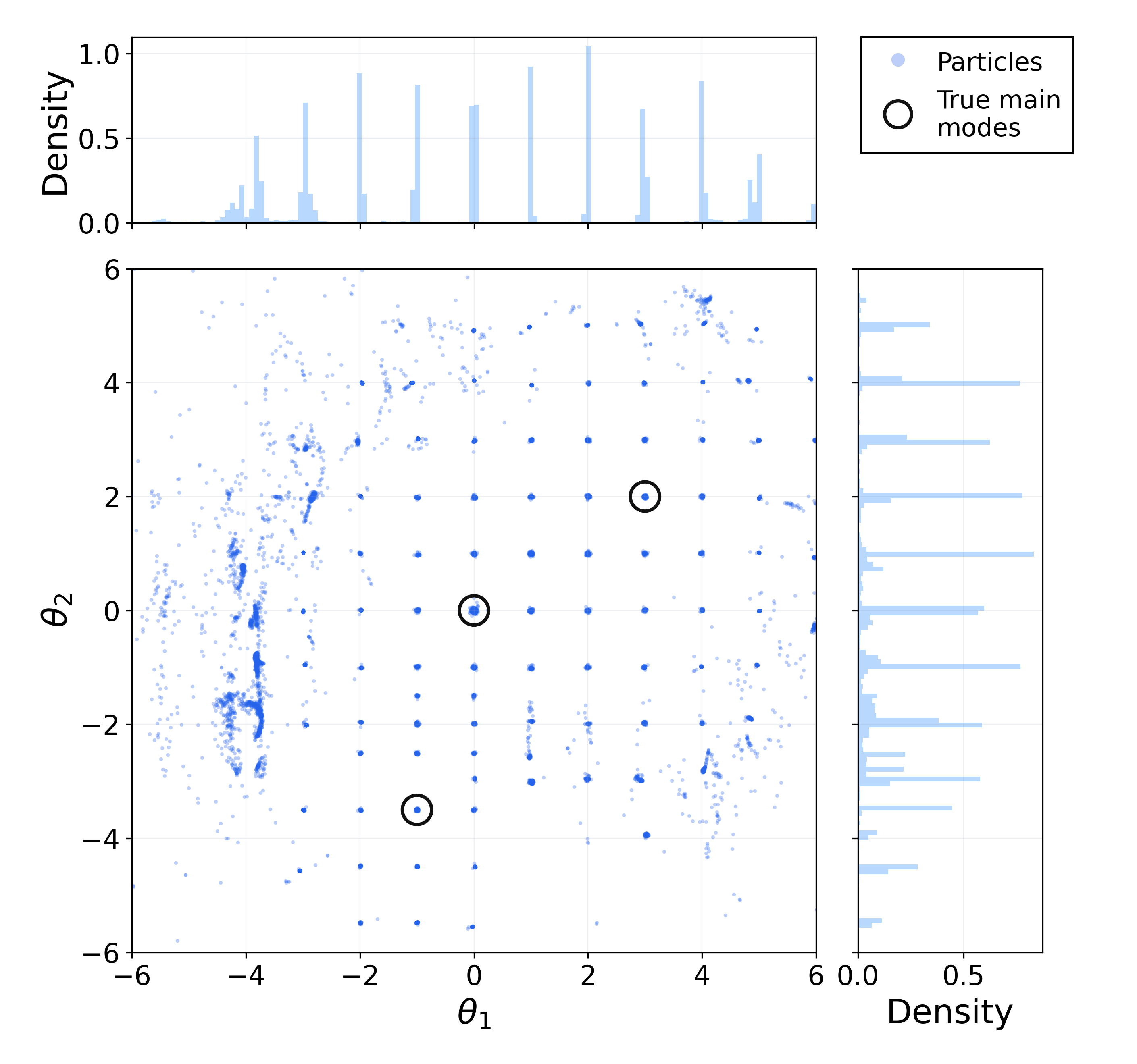}
        \caption{PEKI with broad prior.}
        \label{fig:rastrigin_final_ensemble_peki}
    \end{subfigure}
    
    \vspace{0.5em}
    
    \begin{subfigure}[t]{0.32\textwidth}
        \centering
        \includegraphics[width=\textwidth]{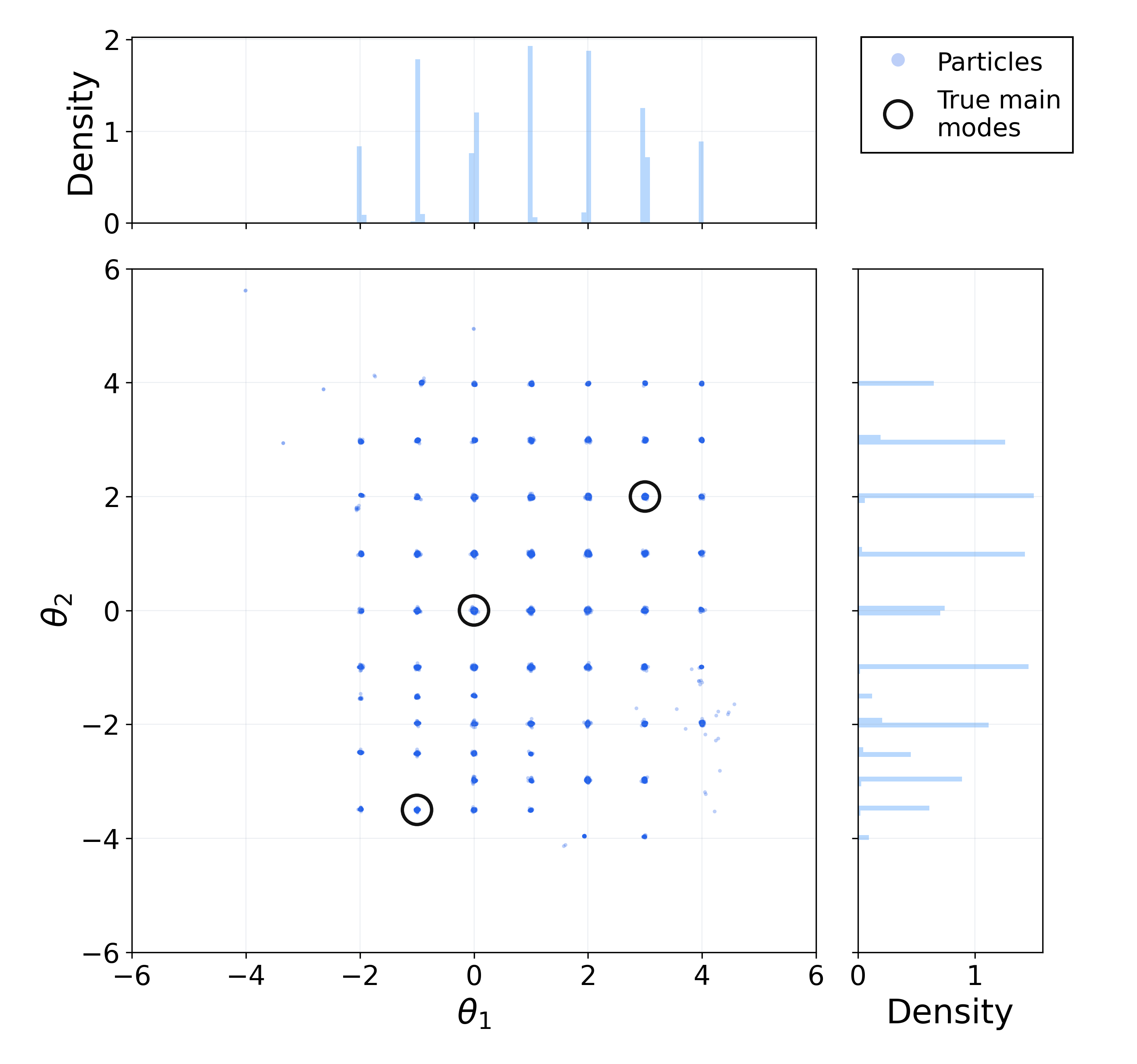}
        \caption{PEKI with compact prior.}
        \label{fig:rastrigin_final_ensemble_peki_small_prior}
    \end{subfigure}
    \hfill
    \begin{subfigure}[t]{0.32\textwidth}
        \centering
        \includegraphics[width=\textwidth]{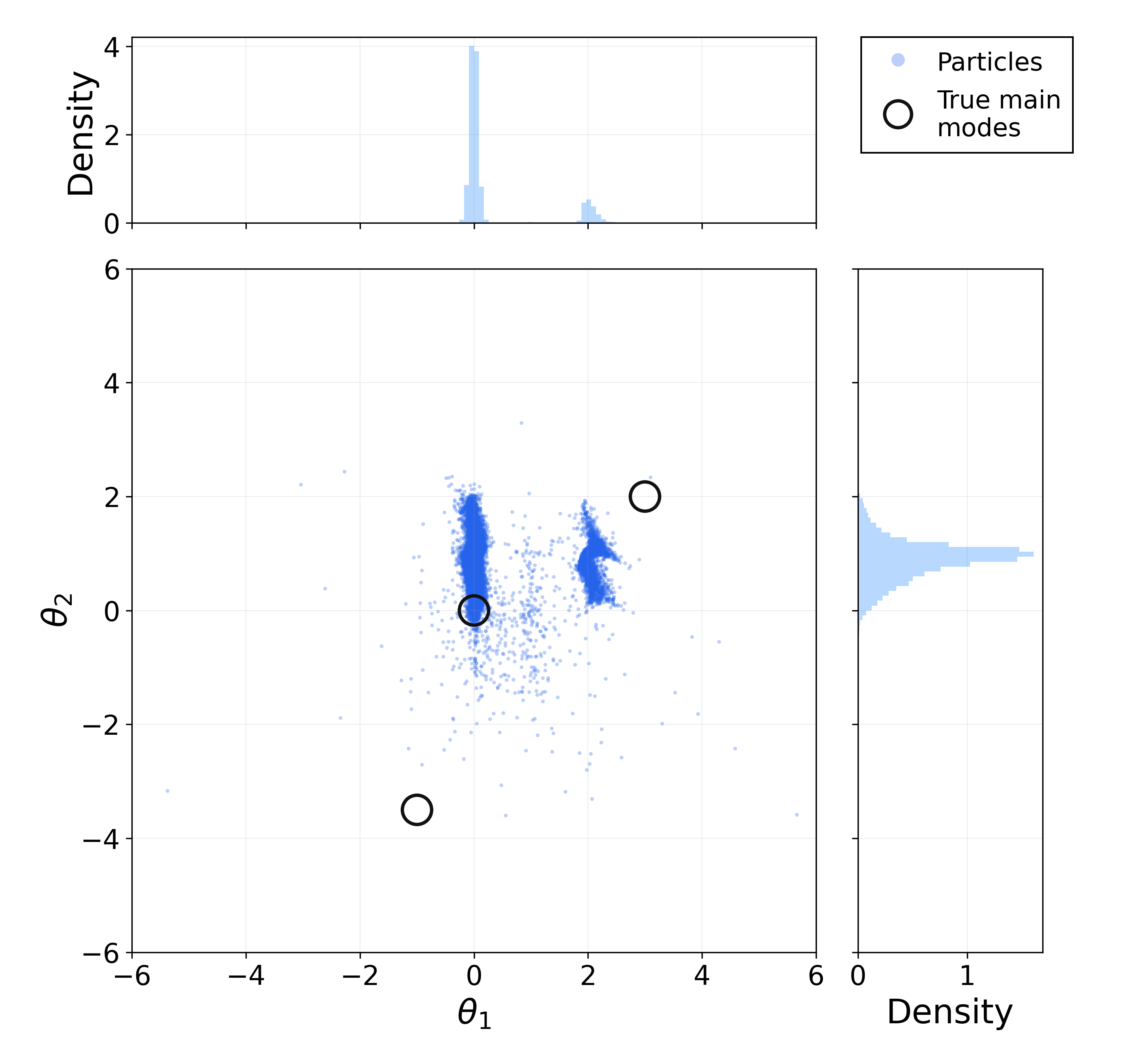}
        \caption{GMM-EKI with compact prior.}
        \label{fig:rastrigin_gmm_eki}
    \end{subfigure}
    \hfill
    \begin{subfigure}[t]{0.32\textwidth}
        \centering
        \includegraphics[width=\textwidth]{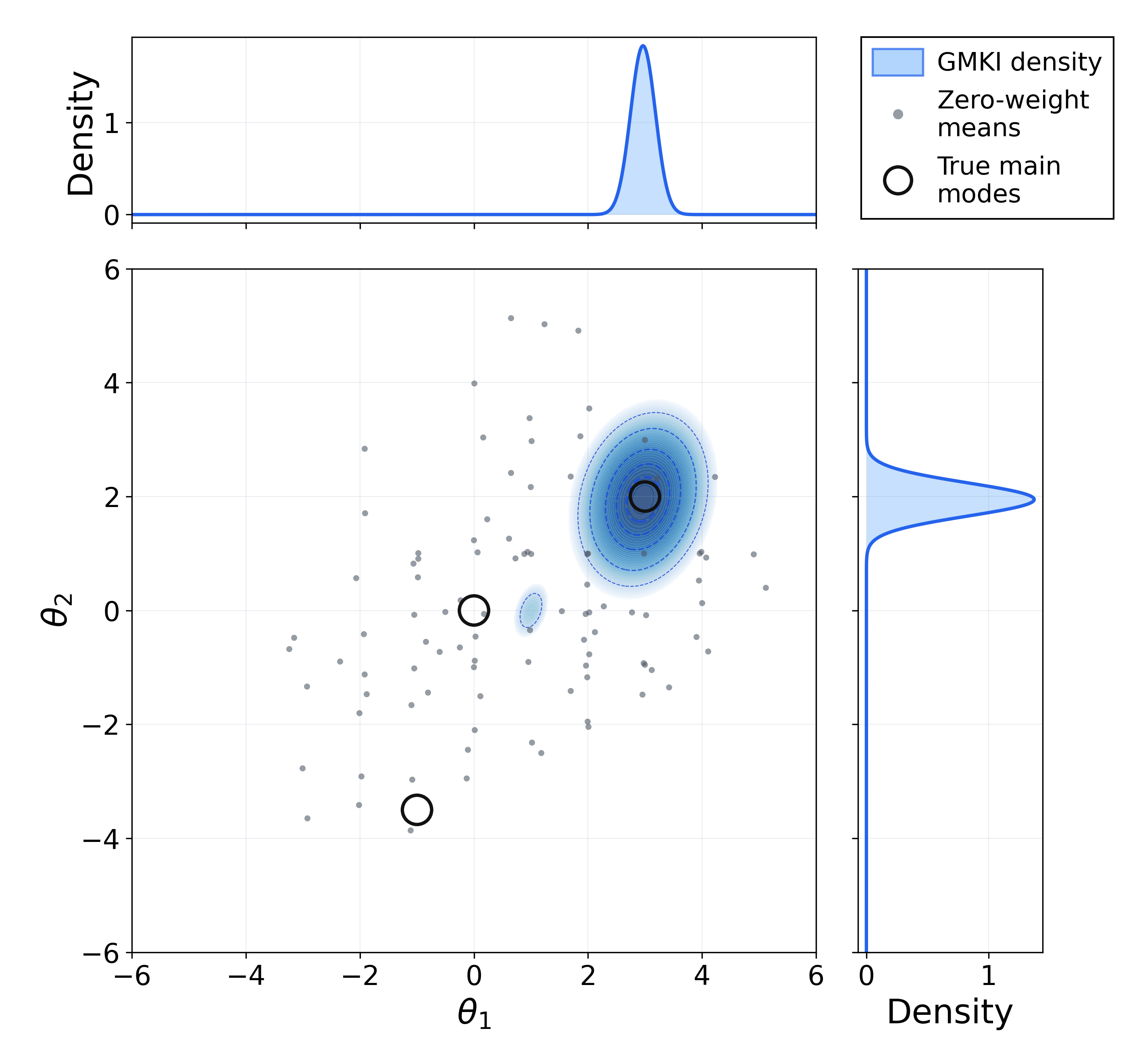}
        \caption{GMKI with compact prior.}
        \label{fig:rastrigin_gmki}
    \end{subfigure}
    
    \caption{Rastrigin inverse problem results. Panels (a), (b) and (c) show the potential landscape, the reference posterior density and its log-posterior landscape. Panels (d) and (e) show final ensembles from standard stochastic EKI, without and with covariance inflation. Panels (f) and (g) show final PEKI ensembles with broad and compact uniform priors. PEKI forms separated particle clusters near all modes. Panels (h) and (i) show the final results via GMM-EKI~\cite{dovera2011multimodal} and GMKI~\cite{chen2024efficient}.}
    \label{fig:rastrigin_results}
\end{figure}

We examine cross-cluster interactions in the Rastrigin example by fixing the ensemble and varying $\kappa$, and evaluate the cross-cluster leakage and the discrepancy between the full and cluster-restricted PEKI updates. Upper panels of Fig.~\ref{fig:rastrigin_propertyr} plot these two quantities, respectively, against $1/\kappa^2$ on logarithmic vertical scales. Both decrease rapidly as $\kappa$ becomes small, consistent with the exponential suppression predicted by Proposition~\ref{prop:cross_cluster_weight_suppression} and Corollary~\ref{cor:cluster_update_discrepancy}.



Unlike the Himmelblau example, however, the curves are not characterized by a single constant slope over the entire small-$\kappa$ range, reflecting two effects. On the one hand, the exponential bound contains an additional, faster-decaying term besides the leading term. On the other hand, the actual leakage and update discrepancy combine contributions from multiple out-of-cluster particles at different distances. At moderate $1/\kappa^2$, several out-of-cluster particles contribute, producing effective decay rates larger than the nearest-particle rate. As $\kappa$ decreases, the contributions associated with more distant particles, together with other faster-decaying corrections, become negligible. The nearest out-of-cluster particle then dominates, and the slopes approach $-(d_{i,r,n}^{\rm out})^2/2$, as shown by the tail reference lines. The curves therefore resolve the finite-$\kappa$ transition toward the leading distance-dependent exponential scale predicted by the non-asymptotic bound.

This transition is more visible in the Rastrigin example because its clusters are more closely spaced. Consequently, contributions from several other clusters remain non-negligible over a wider range of $\kappa$. In the more widely separated Himmelblau example, the leading exponential regime is reached at larger $\kappa$. At still smaller $\kappa$, the Rastrigin update discrepancy reaches numerical precision and develops the observed plateaus.

Finally, we examine the empirical convergence behavior. Starting from the uniform prior on $[-2,4]\times[-4,4]$, we examine convergence over 10 random seeds using the mean distance over all particles, rather than the cluster-wise MSE in Eq.~\eqref{eq:mse error}. The bottom-left panel of Fig.~\ref{fig:rastrigin_propertyr} shows a rapid initial decrease followed by a slower regime, so the raw distance is not a single geometric decay to zero. After excluding the first three iterations, the floor-corrected geometric model $d_k=d_\infty+Cq^k$ gives $R^2\approx0.9993$. The bottom-right log--log panel is nearly linear: fitting $d_k=C(k+1)^{-\alpha}$ over all iterations gives $\alpha\approx0.96$ and $R^2\approx0.9933$, while excluding the first three gives $\alpha\approx1.04$ and $R^2\approx0.9935$.

\begin{figure}[!htbp]
    \centering
    \includegraphics[width=0.7\linewidth]{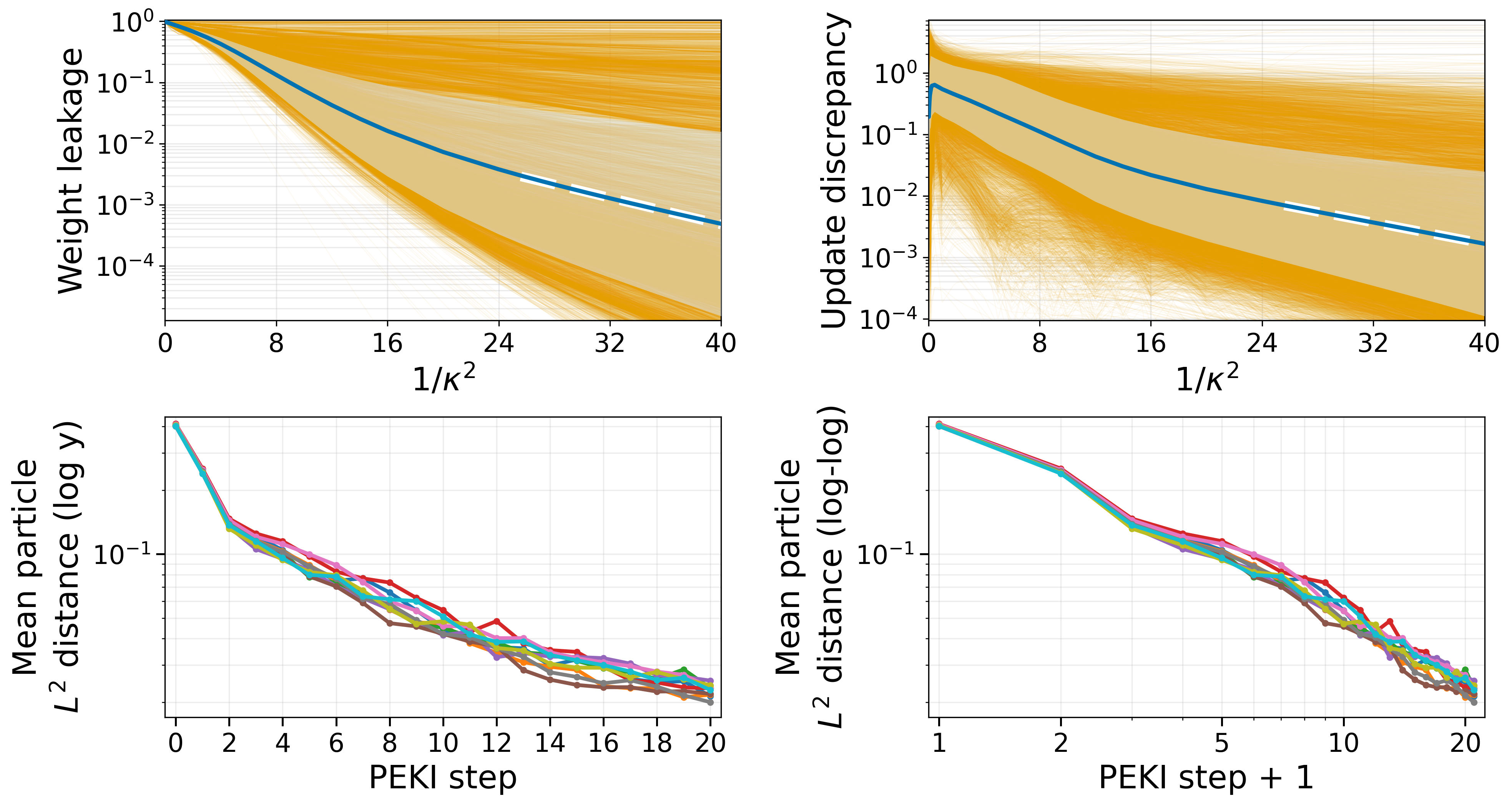}
    \caption{Cross-cluster interactions and empirical error decay in the Rastrigin example. The top-left and top-right subpanels show the cross-cluster weight leakage and update discrepancy, respectively, while the bottom-left and bottom-right subpanels show the convergence of ensemble members to their corresponding low data-misfit solutions on semi-log and log-log scales, respectively.}
    \label{fig:rastrigin_propertyr}
\end{figure}

\subsection{High-dimensional bimodal problem: Navier-Stokes problem}

Following~\cite{chen2024efficient, che2025stable}, we consider the PDE-constrained inverse problem of recovering the initial vorticity field $\omega_0$ from later-time measurements. The flow satisfies the two-dimensional Navier--Stokes equation on $\mathbb{T}^2=[0,2\pi]^2$ in vorticity--streamfunction form:
\[
\begin{aligned}
    &\partial_t \omega + u\cdot\nabla \omega
    = \nu \Delta \omega + \nabla \times f,\\
    &\omega = -\nabla \psi \quad \frac{1}{4\pi^2}\int \psi = 0 \quad u=\left[ \frac{\partial \psi}{\partial x_2},-\frac{\partial \psi}{\partial x_1}\right]^\top + u_b,
\end{aligned}
\]
where the velocity $u$ is solved from the streamfunction $\psi$ with a non-zero mean background $u_b=[0,2\pi]^\top$. $f(x)=[0,\cos(4x_1)]^\top$ is the external forcing.

The spatial symmetry about $x_1=\pi$ is utilized to construct a bimodal problem, observing pointwise vorticity differences:
\[
    y=\omega\left(\left[x_1,x_2\right]^\top\right)-\omega\left(\left[2\pi-x_1,x_2\right]^\top\right) + \eta, \quad \eta \sim\mathcal{N}(0,0.1^2),
\]
at 56 equidistant points, at $T=0.25$ and $T=0.5$, with observation noise. Under this setup, both $\omega_0([x_1,x_2]^\top)$ and $-\omega_0([2\pi-x_1,x_2]^\top)$ will lead to the same observations, resulting in at least two data-consistent solutions.

Rather than the field itself, the unknown parameter is the leading KL/Fourier coefficients of $\omega_0(\theta)$:
\[
    \omega_0([x_1,x_2],\theta)=\sum_{l\in \mathbb{L}} \theta_l^c \sqrt{\lambda_l} \psi_l^c([x_1,x_2])+\theta_l^s \sqrt{\lambda_l} \psi_l^s([x_1,x_2]),
\]
where $\mathbb{L}={(l_x,l_y)|l_x+l_y>0 ~\text{or}~(l_x+l_y=0~\text{and}~l_x>0)}$, and the eigenpairs
\[
    \psi^c_l([x_1,x_2])=\frac{\cos (l \cdot [x_1,x_2])}{\sqrt{2}\pi}, \quad \psi^s_l([x_1,x_2])=\frac{\sin (l \cdot [x_1,x_2])}{\sqrt{2}\pi}, \quad \lambda_l=\frac{1}{|l|^4}.
\]

For computational convenience, the KL expansion is rearranged as a sum over positive integers rather than a lattice:
\[
    \omega_0([x_1,x_2],\theta)=\sum_{l=1}^\infty \theta_l \sqrt{\lambda_l} \psi_l([x_1,x_2]),
\]
where the eigenvalues $\lambda_l$ are in descending order. We truncate the expansion to the first $K$ terms and generate the true vorticity field $\omega_0([x_1,x_2],\theta_\text{ref})$ with $\theta_\text{ref}\in \mathbb{R}^K$. Hence, the parameter in this example will be the coefficients $\theta_\text{ref}$. And there exists a $\theta_\text{ref}^m \neq \theta_\text{ref}$ resulting in a mirrored field and the same observations. The settings so far are exactly the same as \cite{chen2024efficient, che2025stable}. In this work, we consider three truncation dimensions, \(K=10,50,128\), and use different prior families to demonstrate that PEKI is not tied to a particular parametric form of the prior, in contrast to the Gaussian-prior formulations of GMKI \cite{chen2024efficient} and DF-GMVI \cite{che2025stable}. Specifically, we consider a Gaussian prior, a uniform prior, and a two-component Gaussian mixture, as summarized in Table~\ref{tab: NS setup}.


\begin{table}[!htbp]
    \centering
    {\small
    \setlength{\tabcolsep}{4pt}
    \begin{tabular}{c|c|>{\raggedright\arraybackslash}p{0.47\textwidth}|c|c}
    \hline
        Case & $K$ & Prior & \begin{tabular}[c]{@{}c@{}}Ensemble\\size $J$\end{tabular} & \begin{tabular}[c]{@{}c@{}}Polarization\\radius $\kappa$\end{tabular} \\
        \hline
        1 & 10 & Gaussian\newline $\theta\sim\mathcal{N}(0,2\pi^2 I)$ & 20480 & 0.7 \\
        2 & 50 & Uniform\newline $\theta_i\sim\mathcal{U}(c_i-\ell_i/2,c_i+\ell_i/2)$, $i=1,\ldots,K$  & 15360 & 1.5 \\
        3 & 128 & Gaussian mixture\newline $\frac{1}{2}\mathcal{N}(\theta_\text{ref}-0.2,0.5^2 I)+\frac{1}{2}\mathcal{N}(\theta_\text{ref}^m+0.2,0.5^2 I)$ & 5120 & 3 \\
        \hline
    \end{tabular}
    \par\vspace{0.25em}
    \footnotesize
    For Case 2,
    $c_i=(\theta_{\text{ref},i}+\theta_{\text{ref},i}^m)/2$, and
    $\ell_i=\max\{2|\theta_{\text{ref},i}-\theta_{\text{ref},i}^m|,1.5\sqrt{2}\pi\}$.
    \par}
    \caption{Prior distributions and numerical settings for the three NS cases}
    \label{tab: NS setup}
\end{table}

Figure~\ref{fig:NS_results} summarizes the final ensemble marginals and reconstructed field of PEKI in the three NS cases. The left column compares the final ensemble marginals of the displayed leading KL coefficients with the corresponding reference values. Across the three cases, PEKI identifies the true value in most parameter dimensions. In the 50-dimensional case, bimodality is emerging but not fully separated for the 10th coefficient and is not clearly recovered for the 11th and 15th coefficients; the 15th coefficient is also unresolved in the 128-dimensional case. The unresolved dimensions are associated with closely spaced modal values or weak identifiability from the observations.

The right column further compares the recovered modes in the vorticity field. As $K$ increases, the truncated KL expansion retains more terms, and the reference fields exhibit increasingly complex spatial structures, while both reconstructed fields remain visually close to their corresponding reference fields. The accompanying error histories show the relative \(L^2\) field errors of the truth-like and mirror-like modes over the PEKI iterations. In all three cases, both errors decrease rapidly during the first few iterations and then stabilize. After 10 iterations, the two fields have errors of approximately \(5\%\) for \(K=10\), \(10\%\) for \(K=50\), and \(5\%\) for \(K=128\). These errors are lower than the roughly \(20\%\) reported in~\cite{chen2024efficient, che2025stable}, although a direct performance comparison is not intended because the prior settings differ across the methods.

Notably, successful recovery is obtained with ensemble sizes decreasing from \(J=20480\) at \(K=10\) to \(J=5120\) at \(K=128\), despite the increase in parameter dimension. This demonstrates that the ensemble size requirement depends strongly on how the prior covers the relevant solution regions, rather than on parameter dimension alone. We do not seek to determine the minimum ensemble size for each case.

\begin{figure}[!htbp]
    \centering
    \begin{subfigure}[t]{0.48\textwidth}
        \centering
        \includegraphics[width=0.95\textwidth]{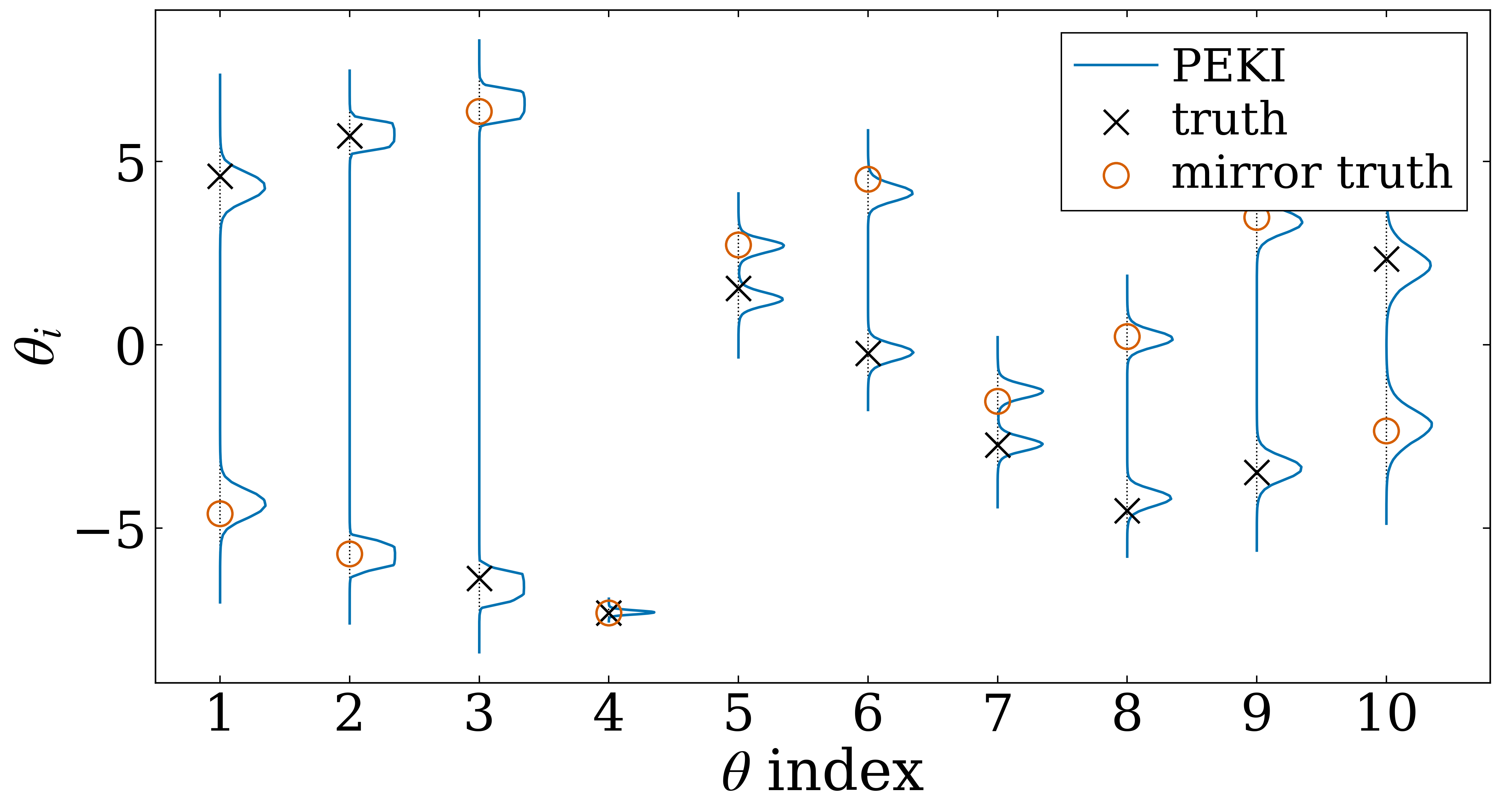}
        \caption{Marginal distribution for 10-dimensional.}
        \label{fig: NS d10 marignal}
    \end{subfigure}
    \hfill
    \begin{subfigure}[t]{0.48\textwidth}
        \centering
        \includegraphics[width=\textwidth]{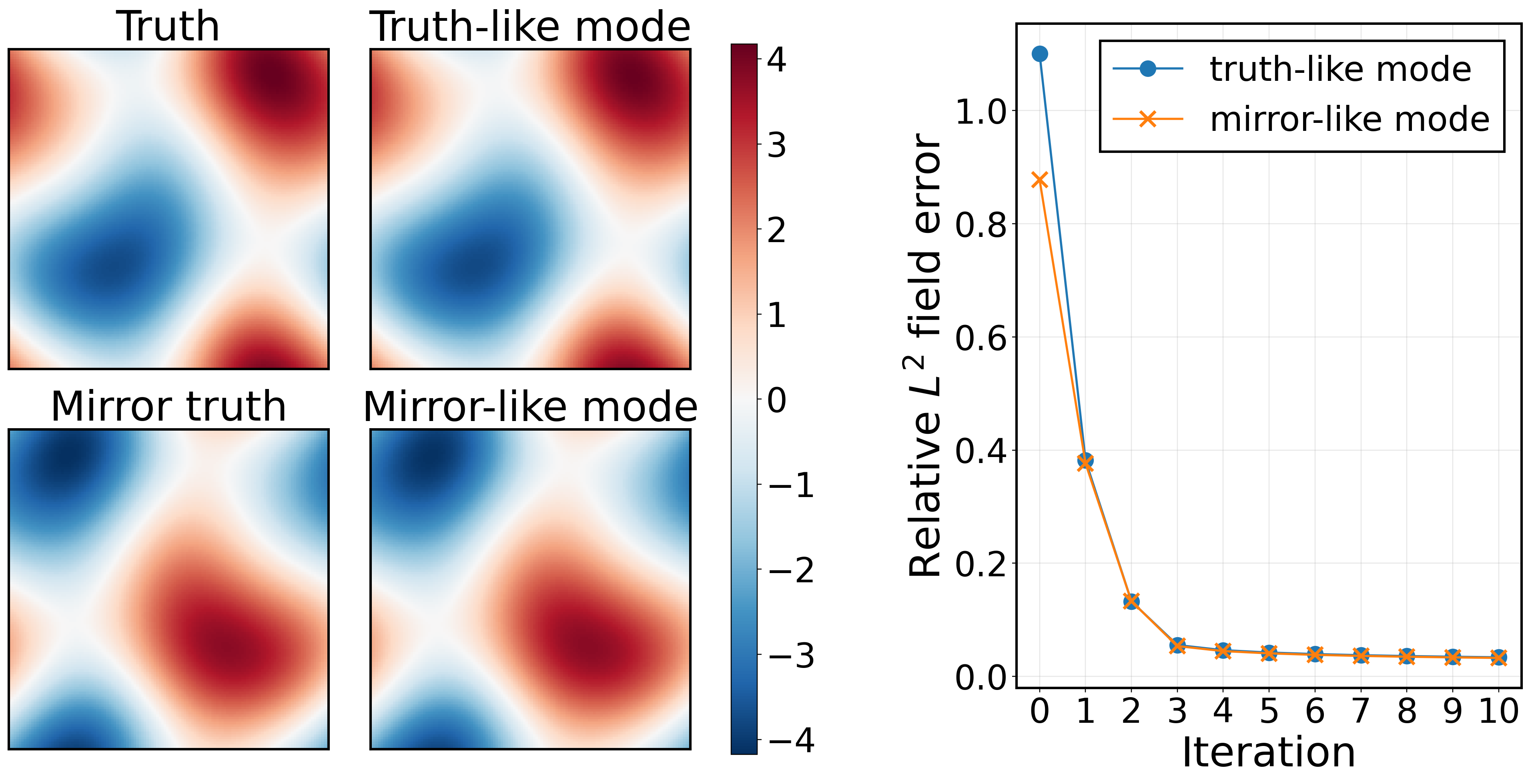}
        \caption{Reconstructed field for 10-dimensional.}
        \label{fig: NS d10 field}
    \end{subfigure}
    
    \vspace{0.5em}
    
    \begin{subfigure}[t]{0.48\textwidth}
        \centering
        \includegraphics[width=0.95\textwidth]{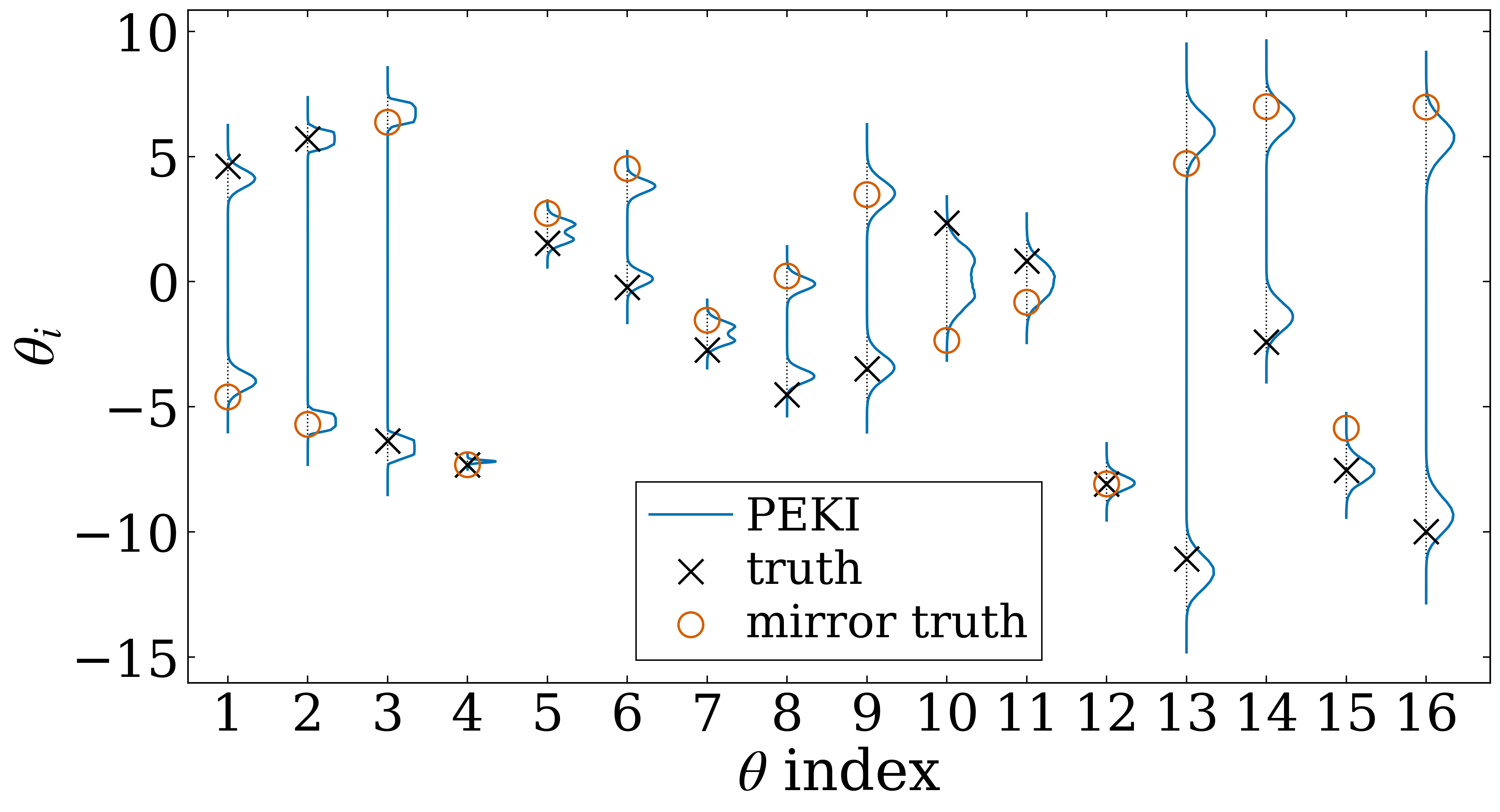}
        \caption{Marginal distribution for 50-dimensional.}
        \label{fig: NS d50 marignal}
    \end{subfigure}
    \hfill
    \begin{subfigure}[t]{0.48\textwidth}
        \centering
        \includegraphics[width=\textwidth]{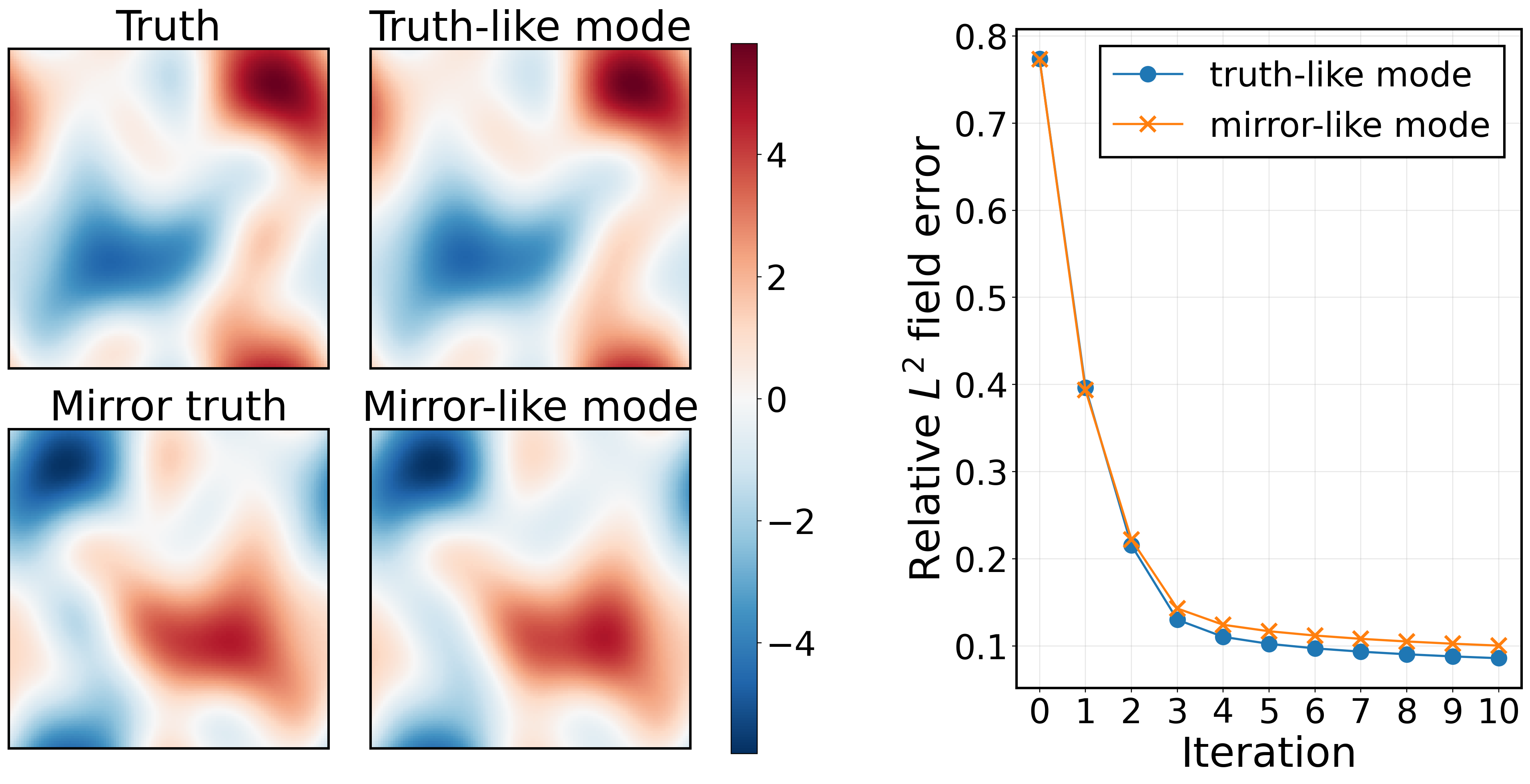}
        \caption{Reconstructed field for 50-dimensional.}
        \label{fig: NS d50 field}
    \end{subfigure}

    \vspace{0.5em}
    
    \begin{subfigure}[t]{0.48\textwidth}
        \centering
        \includegraphics[width=0.95\textwidth]{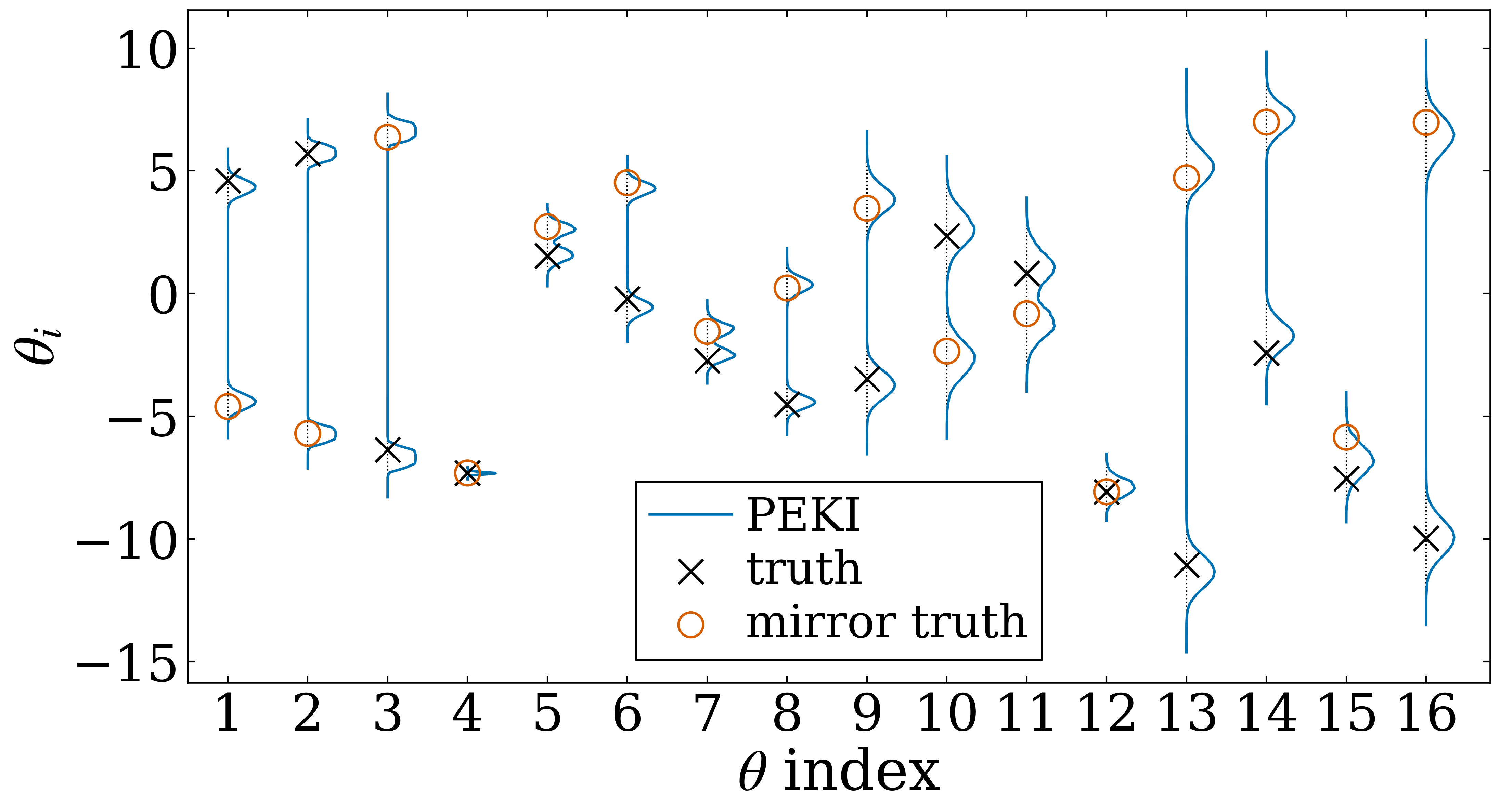}
        \caption{Marginal distribution for 128-dimensional.}
        \label{fig: NS d128 marignal}
    \end{subfigure}
    \hfill
    \begin{subfigure}[t]{0.48\textwidth}
        \centering
        \includegraphics[width=\textwidth]{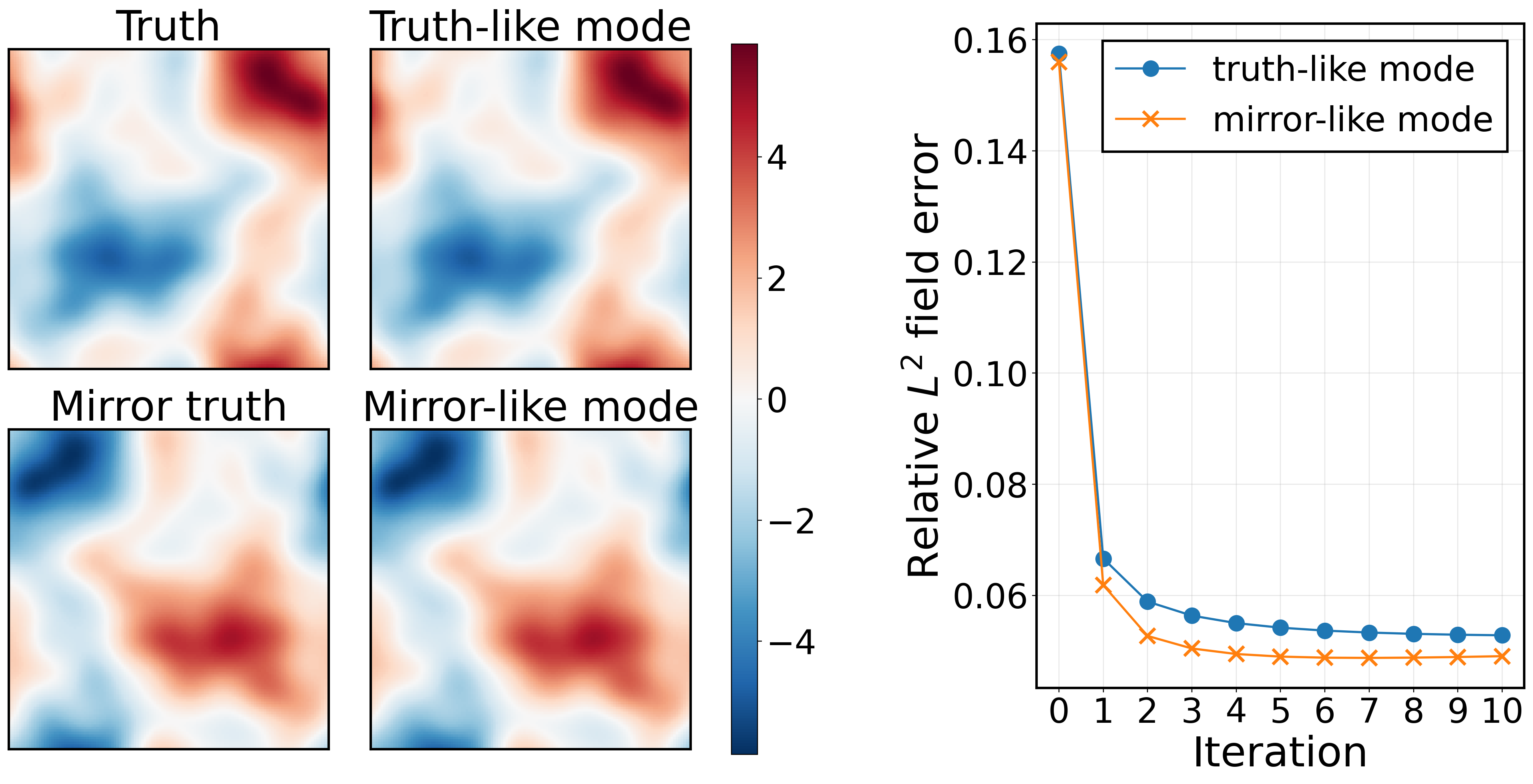}
        \caption{Reconstructed field for 128-dimensional.}
        \label{fig: NS d128 field}
    \end{subfigure}
    
    \caption{PEKI results for the Navier--Stokes inverse problem with \(K=10\) (top), \(K=50\) (middle), and \(K=128\) (bottom). The left column shows the marginals of the displayed KL coefficients and the corresponding truth and mirror values. In each right-hand panel, the \(2\times2\) plots show the reference truth and mirror vorticity fields and their PEKI reconstructions, while the adjacent plot shows the relative \(L^2\) errors of the truth-like and mirror-like reconstructed fields over the PEKI iterations.}
    \label{fig:NS_results}
\end{figure}



\section{Conclusions}
\label{Conclusion}

We developed polarized ensemble Kalman inversion (PEKI) for inverse problems with multiple data-consistent solutions. By replacing the global empirical covariances of standard EKI with particle-dependent, distance-weighted, finite-sample-corrected covariances, PEKI allows different particle clusters to evolve toward distinct solutions while preserving the derivative-free and naturally parallel structure of EKI.

Our finite-ensemble, discrete-time analysis shows that the influence of distant clusters on particle covariances, Kalman gains, and updates is exponentially suppressed as their separation increases relative to the polarization radius. For linear forward maps, PEKI also retains particle-wise residual non-expansion, and the cluster-mean update admits an exact decomposition into a cluster-EKI update, a cross-cluster correction, a within-cluster correction, and an observation-perturbation term. The analysis highlights the main trade-off in polarization-radius selection: a smaller polarization radius can suppress cross-cluster interactions, whereas an excessively small radius may leave too few effective neighbors for stable local covariance estimation.

Based on this analysis, we propose a practical polarization-radius selection procedure whose one-step scan requires no additional forward-model evaluations. The low-dimensional examples demonstrate the practical selection of the radius, provide empirical support for the predicted cross-cluster suppression, and show that the leading distance-dependent exponential scale in the non-asymptotic bound is sharp at the level of the observed update discrepancies. The Navier--Stokes example further demonstrates that PEKI can be applied to high-dimensional PDE-constrained inverse problems while maintaining multiple data-consistent solutions.

Overall, PEKI provides a practical extension of EKI for multimodal inverse problems. It is designed to identify multiple data-consistent solutions rather than recover their relative posterior masses or the full posterior distribution. Future work may therefore consider adaptive or iteration-dependent radii and extensions for recovering posterior information.


\appendix

\section{Proof for suppression of cross-cluster interactions}
\label{app:cluster_decoupling_bounds}
For conciseness, we provide only a brief proof here; a complete and detailed proof is given in the supplementary material, \cref{app:detailed_cluster_decoupling_bounds}.

\begin{lemma}[Pairwise form of the local weighted covariance]
\label{lem:pairwise_weighted_covariance}
For normalized weights \(w_j\geq0\), \(\sum_jw_j=1\), and weighted means \(\bar x=\sum_jw_jx^{(j)}\), \(\bar z=\sum_jw_jz^{(j)}\), the cross-covariance matrix can be rewritten as,
\[
    C^{xz}=\frac{\sum_{j=1}^Jw_j
    (x^{(j)}-\bar x)(z^{(j)}-\bar z)^\top}
    {1-\sum_{j=1}^Jw_j^2}
    =
    \frac{\sum_{1\leq p<q\leq J}w_pw_q
    (x^{(p)}-x^{(q)})(z^{(p)}-z^{(q)})^\top}
    {2\sum_{1\leq p<q\leq J}w_pw_q}.
    \label{eq:pairwise_weighted_covariance_main}
\]
\end{lemma}

\begin{proof}[Proof of \cref{prop:cross_cluster_weight_suppression}]
Set
\[
    q:=\exp\!\left(-\frac{(d_{i,r,n}^{\mathrm{out}})^2}{2\kappa^2}\right),
    \,
    O:=\sum_{j\notin I_r}\widetilde w_{ij,n}=\sum_{j\notin I_r}
    \exp\left(
        -\frac{d_{ij,n}^2}{2\kappa^2}
    \right)  \leq(J-J_r)q      .
\]
Set $    I:=\sum_{j\in I_r}\widetilde w_{ij,n}$. Because $i\in I_r$ and $\widetilde w_{ii,n}=1$, we have $I\geq1$. Hence
\[
    \varepsilon_{i,r,n}=\frac{O}{I+O}
    \leq\frac{(J-J_r)q}{1+(J-J_r)q}
    =\frac{J-J_r}{J-J_r+q^{-1}}
    \leq(J-J_r)q,
\]
which gives both bounds.
\end{proof}

\begin{proof}[Proof of \cref{prop:covariance_gain_decoupling}]
To lighten notation, write 
\[
\begin{aligned}
    &\varepsilon:=\varepsilon_{i,r,n}, \quad w_p:=w_{ip,n}, \quad w_p^{(r)}:=w_{ip,n}^{(r)}, \quad D_r:=D_{i,r,n}=2\sum_{{p<q,\,p,q\in I_r}}w_p^{(r)}w_q^{(r)},\\
    &D_o:=2\sum_{\substack{1\leq p<q\leq J\\\{p,q\}\not\subseteq I_r}}w_pw_q, \quad \Delta\theta_{pq}:=\theta_n^{(p)}-\theta_n^{(q)}, \quad \Delta g_{pq}:=g_n^{(p)}-g_n^{(q)}
\end{aligned}
\]
For $p\in I_r$,
\begin{equation}
    w_p=(1-\varepsilon)w_p^{(r)}.
    \label{eq:full_restricted_weight_relation}
\end{equation}
Define
\[
P_r:=\sum_{\substack{p<q\\p,q\in I_r}}
    w_p^{(r)}w_q^{(r)}\Delta\theta_{pq}\Delta g_{pq}^\top, \quad
    P_o:=\sum_{\substack{1\leq p<q\leq J\\\{p,q\}\not\subseteq I_r}}
    w_pw_q\Delta\theta_{pq}\Delta g_{pq}^\top .
\]
With $a:=(1-\varepsilon)^2$, \cref{lem:pairwise_weighted_covariance} and
\cref{eq:full_restricted_weight_relation} give
\begin{equation}
    C_{i,r,n}^{\theta g}=\frac{P_r}{D_r},
    \qquad
    C_{i,n}^{\theta g}=\frac{aP_r+P_o}{aD_r+D_o}.
    \label{eq:full_restricted_pairwise_forms}
\end{equation}
The total mass outside $I_r$ is $\varepsilon$, and therefore
\[
    D_o
    =2(1-\varepsilon)\varepsilon+\varepsilon^2
      -\sum_{p\notin I_r}w_p^2
    \leq2\varepsilon.
\]
The diameter bounds further imply
\begin{equation}
    \|P_r\|\leq\frac12 R_{\theta,n} R_{g,n} D_r,
    \qquad
    \|P_o\|\leq\frac12 R_{\theta,n} R_{g,n}D_o
    \leq R_{\theta,n} R_{g,n}\varepsilon.
    \label{eq:pairwise_numerator_bounds}
\end{equation}
Subtracting the two expressions in \cref{eq:full_restricted_pairwise_forms}
and using \cref{eq:pairwise_numerator_bounds} yields
\begin{align*}
    C_{i,n}^{\theta g}-C_{i,r,n}^{\theta g}
    &=\frac{D_rP_o-P_rD_o}{D_r(aD_r+D_o)},\\
    \left\|C_{i,n}^{\theta g}-C_{i,r,n}^{\theta g}\right\|
    &\leq
    \frac{D_r\|P_o\|+\|P_r\|D_o}{aD_r^2}
    \leq\frac{2 R_{\theta,n} R_{g,n}}{(1-\varepsilon)^2D_r}\,\varepsilon.
\end{align*}
This proves \cref{eq:c_theta_g_bound}. Replacing $\theta$ by $g$ in the same
argument gives \cref{eq:c_g_g_bound}:
\[
    \left\|C_{i,n}^{gg}-C_{i,r,n}^{gg}\right\|
    \leq\frac{2R_{g,n}^2}{(1-\varepsilon)^2D_r}\,\varepsilon.
\]

It remains to estimate the gains. Let
$H_0:=C_{i,n}^{gg}+R_N$ and $H_r:=C_{i,r,n}^{gg}+R_N$. Since both covariance
matrices are positive semidefinite,
\[
    \|H_0^{-1}\|,\ \|H_r^{-1}\|\leq\mu_N^{-1},
    \qquad
    H_0^{-1}-H_r^{-1}=H_0^{-1}(H_r-H_0)H_r^{-1}.
\]
Moreover, the pairwise form gives
$\|C_{i,r,n}^{\theta g}\|\leq R_{\theta,n} R_{g,n}/2$. The decomposition
\[
    K_{i,n}-K_{i,r,n}
    =(C_{i,n}^{\theta g}-C_{i,r,n}^{\theta g})H_0^{-1}   \quad+C_{i,r,n}^{\theta g}(H_0^{-1}-H_r^{-1})
\]
therefore gives \cref{eq:covariance_gain_decoupling}
\begin{align*}
    \|K_{i,n}-K_{i,r,n}\|
    &\leq \frac1{\mu_N}
    \left\|C_{i,n}^{\theta g}-C_{i,r,n}^{\theta g}\right\|
    +\frac{\|C_{i,r,n}^{\theta g}\|}{\mu_N^2}
    \left\|C_{i,n}^{gg}-C_{i,r,n}^{gg}\right\|\\
    &\leq
    \left(\frac{2R_{\theta,n} R_{g,n}}{\mu_N}
    +\frac{R_{\theta,n} R_{g,n}^3}{\mu_N^2}\right)
    \frac{\varepsilon}{(1-\varepsilon)^2D_r}.
\end{align*}
\end{proof}

\par\smallskip
\begin{proof}[Proof of \cref{cor:cluster_update_discrepancy}]
The update difference satisfies
\[
    \theta_{n+1}^{(i)}-\widetilde\theta_{n+1}^{(i,r)}
    =(K_{i,n}-K_{i,r,n})(y_n^{(i)}-g_n^{(i)}),
\]
so the first bound follows from \cref{eq:covariance_gain_decoupling}. For the
second, set
\[
    q:=\exp\!\left(-\frac{(d_{i,r,n}^{\mathrm{out}})^2}{2\kappa^2}\right),
    \qquad t:=\frac{\varepsilon_{i,r,n}}{1-\varepsilon_{i,r,n}}.
\]
Using the raw weights and $\sum_{j\in I_r}\widetilde w_{ij,n}\geq1$,
\[
    t=
    \frac{\sum_{j\notin I_r}\widetilde w_{ij,n}}
         {\sum_{j\in I_r}\widetilde w_{ij,n}}
    \leq(J-J_r)q.
\]
Therefore
\[
    \frac{\varepsilon_{i,r,n}}{(1-\varepsilon_{i,r,n})^2}
    =t(1+t)
    \leq(J-J_r)q\bigl[1+(J-J_r)q\bigr].
\]
Substitution into the first bound gives the stated non-asymptotic estimate.
\end{proof}


\bibliographystyle{siamplain}
\bibliography{references}

@article{chen2024efficient,
  title={Efficient, multimodal, and derivative-free bayesian inference with Fisher--Rao gradient flows},
  author={Chen, Yifan and Zhengyu Huang, Daniel and Huang, Jiaoyang and Reich, Sebastian and Stuart, Andrew M},
  journal={Inverse Problems},
  volume={40},
  number={12},
  pages={125001},
  year={2024},
  publisher={IOP Publishing}
}

@article{che2025stable,
  title={Stable derivative free Gaussian mixture variational inference for Bayesian inverse problems},
  author={Che, Baojun and Chen, Yifan and Huan, Zhenghao and Huang, Daniel Zhengyu and Wang, Weijie},
  journal={SIAM Journal on Scientific Computing},
  volume={47},
  number={5},
  pages={A2583--A2608},
  year={2025},
  publisher={SIAM}
}

@article{iglesias2013ensemble,
  title={Ensemble Kalman methods for inverse problems},
  author={Iglesias, Marco A and Law, Kody JH and Stuart, Andrew M},
  journal={Inverse Problems},
  volume={29},
  number={4},
  pages={045001},
  year={2013},
  publisher={IOP Publishing}
}

@article{stuart2010inverse,
  title={Inverse problems: a Bayesian perspective},
  author={Stuart, Andrew M},
  journal={Acta numerica},
  volume={19},
  pages={451--559},
  year={2010},
  publisher={Cambridge University Press}
}

@article{kaipio2007statistical,
  title={Statistical inverse problems: discretization, model reduction and inverse crimes},
  author={Kaipio, Jari and Somersalo, Erkki},
  journal={Journal of computational and applied mathematics},
  volume={198},
  number={2},
  pages={493--504},
  year={2007},
  publisher={Elsevier}
}

@article{eldredge2026practical,
  title={Practical Guide to Estimation and Uncertainty Quantification of Aerodynamic Flows},
  author={Eldredge, Jeff D and Mousavi, Hanieh},
  journal={AIAA Journal},
  volume={64},
  number={5},
  pages={2403--2423},
  year={2026},
  publisher={American Institute of Aeronautics and Astronautics}
}

@article{keats2007bayesian,
  title={Bayesian inference for source determination with applications to a complex urban environment},
  author={Keats, Andrew and Yee, Eugene and Lien, Fue-Sang},
  journal={Atmospheric environment},
  volume={41},
  number={3},
  pages={465--479},
  year={2007},
  publisher={Elsevier}
}

@article{eldredge2024bayesian,
  title={Bayesian inference of vorticity in unbounded flow from limited pressure measurements},
  author={Eldredge, Jeff D and Le Provost, Mathieu},
  journal={Journal of Fluid Mechanics},
  volume={986},
  pages={A18},
  year={2024},
  publisher={Cambridge University Press}
}

@article{colton2003inverse,
  title={Inverse Acoustic and Electromagnetic},
  author={Colton, David},
  journal={Inside out: inverse problems and applications},
  volume={47},
  pages={67},
  year={2003},
  publisher={Cambridge University Press}
}

@article{oliver2008inverse,
  title={Inverse theory for petroleum reservoir characterization and history matching},
  author={Oliver, Dean S and Reynolds, Albert C and Liu, Ning},
  journal={Cambridge, New York},
  year={2008}
}

@article{del2006sequential,
  title={Sequential monte carlo samplers},
  author={Del Moral, Pierre and Doucet, Arnaud and Jasra, Ajay},
  journal={Journal of the Royal Statistical Society Series B: Statistical Methodology},
  volume={68},
  number={3},
  pages={411--436},
  year={2006},
  publisher={Oxford University Press}
}

@article{vousden2016dynamic,
  title={Dynamic temperature selection for parallel tempering in Markov chain Monte Carlo simulations},
  author={Vousden, WD and Farr, Will M and Mandel, Ilya},
  journal={Monthly Notices of the Royal Astronomical Society},
  volume={455},
  number={2},
  pages={1919--1937},
  year={2016},
  publisher={Oxford University Press}
}

@article{paulin2019error,
  title={Error bounds for sequential Monte Carlo samplers for multimodal distributions},
  author={Paulin, Daniel and Jasra, Ajay and Thiery, Alexandre},
  journal={Bernoulli},
  volume={25},
  number={1},
  pages={310--340},
  year={2019},
  publisher={JSTOR}
}

@article{haber2000optimization,
  title={On optimization techniques for solving nonlinear inverse problems},
  author={Haber, Eldad and Ascher, Uri M and Oldenburg, Doug},
  journal={Inverse problems},
  volume={16},
  number={5},
  pages={1263--1280},
  year={2000}
}

@article{givoli2021tutorial,
  title={A tutorial on the adjoint method for inverse problems},
  author={Givoli, Dan},
  journal={Computer Methods in Applied Mechanics and Engineering},
  volume={380},
  pages={113810},
  year={2021},
  publisher={Elsevier}
}

@article{sambridge2002monte,
  title={Monte Carlo methods in geophysical inverse problems},
  author={Sambridge, Malcolm and Mosegaard, Klaus},
  journal={Reviews of Geophysics},
  volume={40},
  number={3},
  pages={3--1},
  year={2002},
  publisher={Wiley Online Library}
}

@article{beven2001equifinality,
  title={Equifinality, data assimilation, and uncertainty estimation in mechanistic modelling of complex environmental systems using the GLUE methodology},
  author={Beven, Keith and Freer, Jim},
  journal={Journal of hydrology},
  volume={249},
  number={1-4},
  pages={11--29},
  year={2001},
  publisher={Elsevier}
}

@incollection{marti2025multi,
  title={Multi-start methods},
  author={Mart{\'\i}, Rafael and Lozano, Jose A and Mendiburu, Alexander and Hernando, Leticia},
  booktitle={Handbook of heuristics},
  pages={211--230},
  year={2025},
  publisher={Springer}
}

@article{larson2019derivative,
  title={Derivative-free optimization methods},
  author={Larson, Jeffrey and Menickelly, Matt and Wild, Stefan M},
  journal={Acta Numerica},
  volume={28},
  pages={287--404},
  year={2019},
  publisher={Cambridge University Press}
}

@article{McInnes2017, doi = {10.21105/joss.00205}, url = {https://doi.org/10.21105/joss.00205}, year = {2017}, publisher = {The Open Journal}, volume = {2}, number = {11}, pages = {205}, author = {McInnes, Leland and Healy, John and Astels, Steve}, title = {hdbscan: Hierarchical density based clustering}, journal = {Journal of Open Source Software} }

@article{campello2015hierarchical,
author = {Campello, Ricardo J. G. B. and Moulavi, Davoud and Zimek, Arthur and Sander, J{\"o}rg},
title = {Hierarchical Density Estimates for Data Clustering, Visualization, and Outlier Detection},
year = {2015},
issue_date = {July 2015},
publisher = {Association for Computing Machinery},
address = {New York, NY, USA},
volume = {10},
number = {1},
issn = {1556-4681},
url = {https://doi.org/10.1145/2733381},
doi = {10.1145/2733381},
journal = {ACM Trans. Knowl. Discov. Data},
month = jul,
articleno = {5},
numpages = {51},
}

@article{kovachki2019ensemble,
  title={Ensemble Kalman inversion: a derivative-free technique for machine learning tasks},
  author={Kovachki, Nikola B and Stuart, Andrew M},
  journal={Inverse Problems},
  volume={35},
  number={9},
  pages={095005},
  year={2019},
  publisher={IOP Publishing}
}

@article{rastrigin1974systems,
  title={Systems of extremal control},
  author={Rastrigin, Leonard Andreevi{\v{c}}},
  journal={Nauka},
  year={1974}
}

@article{dovera2011multimodal,
  title={Multimodal ensemble Kalman filtering using Gaussian mixture models},
  author={Dovera, Laura and Della Rossa, Ernesto},
  journal={Computational Geosciences},
  volume={15},
  number={2},
  pages={307--323},
  year={2011},
  publisher={Springer}
}

@article{gryvill2025bayesian,
  title={Bayesian ensemble kalman filter for gaussian mixture models},
  author={Gryvill, H{\aa}kon and Grana, Dario and Tjelmeland, H{\aa}kon},
  journal={Mathematical Geosciences},
  volume={57},
  number={1},
  pages={153--192},
  year={2025},
  publisher={Springer}
}

@article{anderson2024quantile,
  title={A quantile-conserving ensemble filter framework. Part III: Data assimilation for mixed distributions with application to a low-order tracer advection model},
  author={Anderson, Jeffrey and Riedel, Chris and Wieringa, Molly and Ishraque, Fairuz and Smith, Marlee and Kershaw, Helen},
  journal={Monthly Weather Review},
  volume={152},
  number={9},
  pages={2111--2127},
  year={2024},
  publisher={American Meteorological Society}
}

@article{chan2023potential,
  title={The potential benefits of handling mixture statistics via a bi-Gaussian EnKF: Tests with all-sky satellite infrared radiances},
  author={Chan, Man-Yau and Chen, Xingchao and Anderson, Jeffrey L},
  journal={Journal of Advances in Modeling Earth Systems},
  volume={15},
  number={2},
  pages={e2022MS003357},
  year={2023},
  publisher={Wiley Online Library}
}

@inproceedings{grumitt2024flow,
  title={Flow annealed Kalman inversion for gradient-free inference in Bayesian inverse problems},
  author={Grumitt, Richard DP and Karamanis, Minas and Seljak, Uro{\v{s}}},
  booktitle={Physical Sciences Forum},
  volume={9},
  number={1},
  pages={21},
  year={2024},
  organization={MDPI}
}

@article{pasmans2026ensemble,
  title={Ensemble Kalman filter in latent space using a variational autoencoder pair},
  author={Pasmans, Ivo and Chen, Yumeng and Sebastian Finn, Tobias and Bocquet, Marc and Carrassi, Alberto},
  journal={Quarterly Journal of the Royal Meteorological Society},
  volume={152},
  number={775},
  pages={e70070},
  year={2026},
  publisher={Wiley Online Library}
}

@article{zheng2024ensemble,
  title={Ensemble kalman diffusion guidance: A derivative-free method for inverse problems},
  author={Zheng, Hongkai and Chu, Wenda and Wang, Austin and Kovachki, Nikola and Baptista, Ricardo and Yue, Yisong},
  journal={arXiv preprint arXiv:2409.20175},
  year={2024}
}

@article{wacker2026perspectives,
  title={Perspectives on locally weighted ensemble Kalman methods},
  author={Wacker, Philipp},
  journal={Journal of Computational and Applied Mathematics},
  volume={472},
  pages={116756},
  year={2026},
  publisher={Elsevier}
}

@article{reich2021fokker,
  title={Fokker--Planck particle systems for Bayesian inference: Computational approaches},
  author={Reich, Sebastian and Weissmann, Simon},
  journal={SIAM/ASA Journal on Uncertainty Quantification},
  volume={9},
  number={2},
  pages={446--482},
  year={2021},
  publisher={SIAM}
}

@article{reimann2026robust,
  title={Robust ensemble Kalman filtering under observation noise misspecification via diffusion score matching},
  author={Reimann, Hans and Reich, Sebastian},
  journal={arXiv preprint arXiv:2605.26881},
  year={2026}
}

@article{zhou2026neural,
  title={Neural ensemble Kalman filter: Data assimilation for compressible flows with shocks},
  author={Zhou, Xu-Hui and Beronilla, Lorenzo and Sleeman, Michael K and Hu, Hangchuan and Morzfeld, Matthias and Stuart, Andrew M and Zaki, Tamer A},
  journal={Journal of Computational Physics},
  pages={115136},
  year={2026},
  publisher={Elsevier}
}

@article{ding2020ensemble,
  title={Ensemble Kalman inversion for nonlinear problems: weights, consistency, and variance bounds},
  author={Ding, Zhiyan and Li, Qin and Lu, Jianfeng},
  journal={arXiv preprint arXiv:2003.02316},
  year={2020}
}

@article{subrahmanya2026copula,
  title={A copula-based rank histogram ensemble filter},
  author={Subrahmanya, Amit N and Bessac, Julie and Popov, Andrey A and Sandu, Adrian},
  journal={International Journal for Uncertainty Quantification},
  volume={16},
  number={1},
  year={2026},
  publisher={Begel House Inc.}
}

@article{che2026adaptive,
  title={Adaptive Exponential Integration for Stable Gaussian Mixture Black-Box Variational Inference},
  author={Che, Baojun and Chen, Yifan and Huang, Daniel Zhengyu and Mao, Xinying and Wang, Weijie},
  journal={arXiv preprint arXiv:2601.14855},
  year={2026}
}

@article{frei2013mixture,
  title={Mixture ensemble Kalman filters},
  author={Frei, Marco and K{\"u}nsch, Hans R},
  journal={Computational Statistics \& Data Analysis},
  volume={58},
  pages={127--138},
  year={2013},
  publisher={Elsevier}
}

@article{frei2013bridging,
  title={Bridging the ensemble Kalman and particle filters},
  author={Frei, Marco and K{\"u}nsch, Hans R},
  journal={Biometrika},
  volume={100},
  number={4},
  pages={781--800},
  year={2013},
  publisher={Oxford University Press}
}

@article{chustagulprom2016hybrid,
  title={A hybrid ensemble transform particle filter for nonlinear and spatially extended dynamical systems},
  author={Chustagulprom, Nawinda and Reich, Sebastian and Reinhardt, Maria},
  journal={SIAM/ASA Journal on Uncertainty Quantification},
  volume={4},
  number={1},
  pages={592--608},
  year={2016},
  publisher={SIAM}
}

@article{shen2015modified,
  title={A modified ensemble Kalman particle filter for non-Gaussian systems with nonlinear measurement functions},
  author={Shen, Zheqi and Tang, Youmin},
  journal={Journal of Advances in Modeling Earth Systems},
  volume={7},
  number={1},
  pages={50--66},
  year={2015},
  publisher={Wiley Online Library}
}

@article{papadakis2010data,
  title={Data assimilation with the weighted ensemble Kalman filter},
  author={Papadakis, Nicolas and M{\'e}min, Etienne and Cuzol, Anne and Gengembre, Nicolas},
  journal={Tellus A: Dynamic Meteorology and Oceanography},
  volume={62},
  number={5},
  pages={673--697},
  year={2010},
  publisher={Taylor \& Francis}
}

@article{zhang2018iterative,
  title={An iterative local updating ensemble smoother for estimation and uncertainty assessment of hydrologic model parameters with multimodal distributions},
  author={Zhang, Jiangjiang and Lin, Guang and Li, Weixuan and Wu, Laosheng and Zeng, Lingzao},
  journal={Water Resources Research},
  volume={54},
  number={3},
  pages={1716--1733},
  year={2018},
  publisher={Wiley Online Library}
}

@article{anderson2003local,
  title={A local least squares framework for ensemble filtering},
  author={Anderson, Jeffrey L},
  journal={Monthly weather review},
  volume={131},
  number={4},
  pages={634--642},
  year={2003}
}

@article{metref2014non,
  title={A non-Gaussian analysis scheme using rank histograms for ensemble data assimilation},
  author={Metref, Sammy and Cosme, Emmanuel and Snyder, Chris and Brasseur, Pierre},
  journal={Nonlinear Processes in Geophysics},
  volume={21},
  number={4},
  pages={869--885},
  year={2014},
  publisher={Copernicus Publications G{\"o}ttingen, Germany}
}

@article{mo2019deep,
  title={Deep autoregressive neural networks for high-dimensional inverse problems in groundwater contaminant source identification},
  author={Mo, Shaoxing and Zabaras, Nicholas and Shi, Xiaoqing and Wu, Jichun},
  journal={Water Resources Research},
  volume={55},
  number={5},
  pages={3856--3881},
  year={2019},
  publisher={Wiley Online Library}
}

@article{zheng2023efficient,
  title={Efficient estimation of groundwater contaminant source and hydraulic conductivity by an ILUES framework combining GAN and CNN},
  author={Zheng, Na and Jiang, Simin and Xia, Xuemin and Kong, Weiming and Li, Zhi and Gu, Simin and Wu, Zijun},
  journal={Journal of Hydrology},
  volume={621},
  pages={129677},
  year={2023},
  publisher={Elsevier}
}

@article{zhang2024integration,
  title={Integration of DDPM and ILUES for simultaneous identification of contaminant source parameters and non-Gaussian channelized hydraulic conductivity field},
  author={Zhang, Xun and Jiang, Simin and Zheng, Na and Xia, Xuemin and Li, Zhi and Zhang, Ruicheng and Zhang, Jiangjiang and Wang, Xinshu},
  journal={Water Resources Research},
  volume={60},
  number={9},
  pages={e2023WR036893},
  year={2024},
  publisher={Wiley Online Library}
}

@article{bungert2025polarized,
  title={Polarized consensus-based dynamics for optimization and sampling},
  author={Bungert, Leon and Roith, Tim and Wacker, Philipp},
  journal={Mathematical Programming},
  volume={211},
  number={1},
  pages={125--155},
  year={2025},
  publisher={Springer}
}

@article{gjini2026ensemble,
  title={The ensemble kalman inversion race},
  author={Gjini, Rebecca and Morzfeld, Matthias and Dunbar, Oliver RA and Schneider, Tapio},
  journal={Journal of Advances in Modeling Earth Systems},
  volume={18},
  number={9},
  pages={e2025MS005629},
  year={2026},
  publisher={Wiley Online Library}
}
\end{document}